\documentclass{amsart}

\usepackage{amsmath,amssymb}
\usepackage{amsthm}
\usepackage{enumerate}
\usepackage{wasysym}

\usepackage{color}
\definecolor{blue}{RGB}{49,40,166}

\newcommand{\cF}{\mathcal{F}}
\newcommand{\cG}{\mathcal{G}}
\newcommand{\cI}{\mathcal{I}}
\newcommand{\cK}{\mathcal{K}}
\newcommand{\cL}{\mathcal{L}}
\newcommand{\cM}{\mathcal{M}}
\newcommand{\cO}{\mathcal{O}}
\newcommand{\cU}{\mathcal{U}}
\newcommand{\claim}{\hfill$\dashv_{\text{\scriptsize{claim}}}$}
\newcommand{\seq}{\subseteq}
\newcommand{\vphi}{\varphi}
\newcommand{\func}{\longrightarrow}
\newcommand{\miff}{\makebox[.4in]{$\Leftrightarrow$}}
\newcommand{\mand}{\makebox[.4in]{and}}
\newcommand{\bbar}{\bar{b}}
\newcommand{\Fraisse}{Fra\"{i}ss\'{e}} 
\newcommand{\thrn}{\text{\thorn}}
\def\G{\mathbb G}
\def\M{\mathbb M}
\def\Im{\operatorname{Im}}
\def\Aut{\operatorname{Aut}}
\def\dcl{\operatorname{dcl}}
\def\acl{\operatorname{acl}}
\def\tp{\operatorname{tp}}
\def\SOP{\operatorname{SOP}}
\def\TP{\operatorname{TP}}
\def\NSOP{\operatorname{NSOP}}
\def\NTP{\operatorname{NTP}}
\def\Th{\operatorname{Th}}
\def\EM{\operatorname{EM}}
\def\eq{\operatorname{eq}}
\def\heq{\operatorname{heq}}
\def\eacl{\operatorname{acl}^{\operatorname{eq}}}
\def\edcl{\operatorname{dcl}^{\operatorname{eq}}}
\def\bdd{\operatorname{bdd}}

\def\Ind{\setbox0=\hbox{$x$}\kern\wd0\hbox to 0pt{\hss$\mid$\hss}
\lower.9\ht0\hbox to 0pt{\hss$\smile$\hss}\kern\wd0}

\def\Notind{\setbox0=\hbox{$x$}\kern\wd0\hbox to 0pt{\mathchardef
\nn=12854\hss$\nn$\kern1.4\wd0\hss}\hbox to
0pt{\hss$\mid$\hss}\lower.9\ht0 \hbox to 0pt{\hss$\smile$\hss}\kern\wd0}

\def\ind{\mathop{\mathpalette\Ind{}}}

\def\nind{\mathop{\mathpalette\Notind{}}}

\newcommand{\inda}{\ind^{\!\!\textnormal{a}}}
\newcommand{\ninda}{\nind^{\!\!\textnormal{a}}}
\newcommand{\indiv}{\ind^{\!\!\textnormal{d}}}
\newcommand{\nindiv}{\nind^{\!\!\textnormal{d}}}
\newcommand{\indf}{\ind^{\!\!\textnormal{f}}}

\makeatletter
\let\@cleartopmattertags\relax
\newcommand\articleend
 {\enddoc@text
  \let\authors\@empty
  \let\shortauthors\@empty
  \let\contribs\@empty
  \let\xcontribs\@empty
  \let\toccontribs\@empty
  \let\addresses\@empty
  \let\thankses\@empty
  \newpage
 }
\let\@wraptoccontribs\wraptoccontribs
\makeatother

\newtheorem{theorem}{Theorem}[section]
\newtheorem{lemma}[theorem]{Lemma}
\newtheorem{corollary}[theorem]{Corollary}
\newtheorem{proposition}[theorem]{Proposition}
\newtheorem{fact}[theorem]{Fact}

\newtheorem*{theorem*}{Theorem}

\theoremstyle{definition}
\newtheorem{definition}[theorem]{Definition}
\newtheorem{example}[theorem]{Example}
\newtheorem{remark}[theorem]{Remark}
\newtheorem{question}[theorem]{Question}

\AtBeginDocument{%
   \def\MR#1{}
}

\author{Gabriel Conant}

\begin{document}

\title{An axiomatic approach to free amalgamation}

%\author{Gabriel Conant}

\address{Department of Mathematics\\
University of Notre Dame\\
Notre Dame, IN, 46656, USA}

\email{gconant@nd.edu}

\date{October 13, 2016}

\begin{abstract}
We use axioms of abstract ternary relations to define the notion of a \emph{free amalgamation theory}. These form a subclass of first-order theories, without the strict order property, encompassing many prominent examples of countable structures in relational languages, in which the class of algebraically closed substructures is closed under free amalgamation. We show that any free amalgamation theory has elimination of hyperimaginaries and weak elimination of imaginaries. With this result, we use several families of well-known homogeneous structures to give new examples of rosy theories. We then prove that, for free amalgamation theories, simplicity coincides with $\NTP_2$ and, assuming modularity, with $\NSOP_3$ as well. We also show that any simple free amalgamation theory is $1$-based. Finally, we prove a combinatorial characterization of simplicity for \Fraisse\ limits with free amalgamation, which provides new context for the fact that the generic $K_n$-free graphs are $\SOP_3$, while the higher arity generic $K^r_n$-free $r$-hypergraphs are simple.\medskip

\noindent \textcolor{blue}{Author's note (11/2/2023): This version of the paper has been updated from the published version in order to address a number of important corrections in Section \ref{sec:simple}. To preserve the historical record, brief comments have been added to places where there is an error, with references to the corrigenda attached at the end of this version (starting on page 24). Some minor typos in the main article have been corrected without explicit mention, and the citation list has been updated.}
\end{abstract}

\maketitle

\section{Introduction}

In the classification of unstable first-order theories, the dividing lines given by $\TP_2$ and $\SOP_3$ have consistently thwarted progress in understanding general structural behavior for theories without the strict order property (i.e. $\NSOP$ theories). On the other hand, all of the known examples of $\NSOP$ theories are either simple or have $\TP_2$; and many (if not most) non-simple examples have $\SOP_3$ as well. Whether these observations will lead to general theorems remains an intriguing open question. The goal of this paper is to develop structural results for a general subclass of $\NSOP$ theories, called \emph{free amalgamation theories}, which are defined by the existence of an abstract ternary notion of independence resembling free amalgamation in relational structures. This subclass will include many well-established examples of theories which are either simple or have $\SOP_3$ and $\TP_2$. The canonical examples are \Fraisse\ limits, closed under free amalgamation, such as the random graph (or \emph{Rado graph}) and generic $K_n$-free graphs (or \emph{Henson graphs}). Other examples, in which free amalgamation is more restricted, include the generic $(K_n+K_3)$-free graphs constructed by Komj\'{a}th \cite{Kom} and Cherlin, Shelah, and Shi \cite{CSS}, as well as a small class of well-behaved Hrushovski constructions. 

The reason for focusing on free amalgamation theories is that a significant majority of the known examples of non-simple $\NSOP$ theories are ``generic" structures with a high level of homogeneity. At present, it is still unclear how to precisely distill the nature of $\NSOP$ homogeneous structures. However, our results will show that homogeneity arising from \emph{free} amalgamation has significant consequences for the model theory of the structure. Moreover, the essential features of free amalgamation can be described by a model theoretic axiomatic framework, which allows cumbersome syntactic analysis to be replaced by smoother ``geometric" arguments. There is currently only one other axiomatic framework which includes examples of $\NSOP$ theories with $\TP_2$ and $\SOP_3$, namely, thorn-forking in rosy theories. However, the class of rosy theories is quite broad, and rosiness alone does not imply the specific instances of good model theoretic behavior that we will obtain here for free amalgamation theories. 

The main results are as follows. We first verify that free amalgamation theories are indeed a subclass of $\NSOP$ theories. In particular, using a similar argument as in unpublished work of Patel \cite{PaSOP4}, we give a short proof that any free amalgamation theory is $\NSOP_4$ (Theorem \ref{thm:NSOP4}). This generalizes Patel's methods to the axiomatic framework, and crystallizes the frequently observed connection between free amalgamation and $\NSOP_4$. This result also overlaps with work of Evans and Wong \cite{EvWo} on  Hrushovski constructions, and  of Shelah and Usvyatsov \cite{SUgroups} on groups.

 We then show that any free amalgamation theory has elimination of hyperimaginaries and weak elimination of imaginaries (see Theorem \ref{thm:WEI}). Using this, we provide new examples of rosy theories, including the class of \Fraisse\ limits closed under free amalgamation, which are superrosy of $U^\thrn$-rank $1$. We also show that the generic $(K_n+K_3)$-free graphs are superrosy of $U^\thrn$-rank $2$ (see Theorem \ref{thm:bowtie}). 

Finally, we analyze the role of simplicity in free amalgamation theories. We show that simplicity coincides with $\NTP_2$ and also with the equivalence of nonforking and algebraic independence (see Theorem \ref{thm:simp}). As a corollary, it follows that any simple free amalgamation theory is modular (in the sense of \cite{Adgeo}). Using the results above on (hyper)imaginaries, we then show that any simple free amalgamation theory is $1$-based (see Corollary \ref{cor:1B}). We also prove that, for modular free amalgamation theories, simplicity coincides with $\NSOP_3$ (see Theorem \ref{thm:NSOP3}). In particular, modular free amalgamation theories form the first example of a general, axiomatically defined class of first-order theories, in which we (nontrivially) obtain the equivalence of simplicity, $\NTP_2$, and $\NSOP_3$ (which, as previously noted, seems to be a much broader phenomenon).

For our main class of motivating examples, this results in the following fairly complete analysis of model theoretic behavior.

\begin{theorem}
Suppose $\cM$ is a countable ultrahomogeneous structure, in a finite relational language, whose age is closed under free amalgamation of $\cL$-structures. Let $T=\Th(\cM)$.
\begin{enumerate}[$(a)$]
\item $T$ has elimination of hyperimaginaries and weak elimination of imaginaries, and is rosy with $U^\thrn(T)=1$.
\item $T$ is $\NSOP_4$. Moreover, the following are equivalent.
\begin{enumerate}[$(i)$]
\item $T$ is simple.
\item $T$ is $\NTP_2$.
\item $T$ is $\NSOP_3$.
\end{enumerate}
\item If $T$ is simple then it is supersimple, of $SU$-rank $1$, and $1$-based.
\end{enumerate}
\end{theorem}

The statements in this theorem are consequences of the various main results in this paper, which are shown for the more general class of free amalgamation theories (Definition \ref{def:FAT}). Given $\cM$ as in the theorem, the justification that $\Th(\cM)$ is a free amalgamation theory is done in Proposition \ref{prop:ex}. Several parts of the theorem also require the observation that $\Th(\cM)$ is modular, which follows from the fact that algebraic closure in $\cM$ is trivial (see Proposition \ref{prop:alg}$(d)$). Part $(a)$, which answers questions posed to us by Cameron Hill and Vera Koponen, combines Theorem \ref{thm:WEI}, Corollary \ref{cor:FArosy}, and Proposition \ref{prop:triv}. Part $(b)$ combines Theorems \ref{thm:NSOP4}, \ref{thm:simp}, and \ref{thm:NSOP3}. Part $(c)$ uses Corollary \ref{cor:1B} to conclude $T$ is $1$-based, and uses the description of forking given by Theorem \ref{thm:simp} to conclude $T$ is supersimple SU-rank $1$ (it is also a general fact that $U$-rank and $U^{\thrn}$-rank coincide for supersimple theories \cite[Theorem 5.1.4]{OnTH}). Part $(c)$ also provides progress toward a question of Koponen \cite{Kop1B} on whether \emph{any} countable, simple, ultrahomogeneous structure, in a finite relational language, is $1$-based. We again emphasize that $\NSOP_4$ in part $(b)$ was first proved by Patel \cite{PaSOP4}. After obtaining our results, we later found that weak elimination of imaginaries in part $(a)$ also follows from \cite[Lemma 2.7]{MacTe}.

The final result of the paper, Theorem \ref{thm:irred}, is a combinatorial characterization of simplicity for $\Th(\cM)$, given in terms of \emph{irreducibility} of forbidden substructures, for certain $\cM$ as in the theorem above. The proof uses a generalization of a result of Hrushovski \cite{Udibook} on the generic $K^r_n$-free $r$-hypergraphs (with $r>2$). 

\subsection*{Acknowledgements} 

I would like to thank David Evans, Vera Koponen, Alex Kruckman, Maryanthe Malliaris, and Rehana Patel for many fruitful conversations. I also thank the referee for several improvements to the exposition, and Vera Koponen for pointing out an error in an earlier version.

\section{Notation and Definitions}

Fix a complete first-order theory $T$ and a $\kappa$-saturated monster model $\M$ of $T$, for $\kappa$ sufficiently large. We write $A\subset\M$ to mean $A\seq\M$ and $|A|<\kappa$. Given $A,B\subset\M$, we let $AB$ denote $A\cup B$. We use singletons $a,b,c,x,y,z,\ldots$ to denote tuples of length $<\kappa$. Given a tuple $a$, we let $\ell(a)$ denote the length (or domain) of $a$ and, abusing notation, we identify $a$ with the subset of $\M$ given by the range of $a$. We write $a\in\M$ to mean $a$ is a tuple of elements from $\M$. When the domain of the tuple is important, we may emphasize this by writing $a\in\M^I$. Given an automorphism $\sigma\in\Aut(\M)$ and a tuple $a=(a_i:i\in I)\in\M^I$, we let $\sigma(a)$ denote the tuple $(\sigma(a_i):i\in I)$. Given tuples $a,b\in\M$, and $C\subset\M$, we write $a\equiv_C b$ if $a,b\in\M^I$, for some common domain $I$, and $\sigma(a)=b$ for some $\sigma\in\Aut(\M/C)$.

In many cases, we index \emph{sequences} of tuples with subscripts (e.g. $(a_i)_{i<\omega}$, where each $a_i$ is a tuple). Therefore, in situations where we also want to reference the specific coordinates of the tuples in such a sequence, we will use superscripts to index tuples and subscripts to index coordinates (e.g. $(a^l)_{l<\omega}$ with $a^l=(a^l_i:i\in I)$).

Suppose $a\in\M^I$ is a tuple with domain $I$. A \textit{subtuple} of $a$ is a tuple of the form $a_J:=(a_i:i\in J)$, where $J$ is a subset of $I$. We write $c\sqsubset a$ to denote that $c$ is a subtuple of $a$. Given an indiscernible sequence $\cI=(a^l)_{l<\omega}$, we define the \textit{common intersection} of $\cI$ to be the (possibly empty) subtuple $a^0_J\sqsubset a^0$, where $J=\{i\in I:a^0_i=a^1_i\}$. 

Let $\acl$ denote algebraic closure in $\M$; $A\subset\M$ is \textit{closed} if $\acl(A)=A$. We say:
\begin{enumerate}[(1)]
\item  $\acl$ is \textit{locally finite} if $\acl(A)$ is finite for all finite $A\subset\M$;
\item $\acl$ is  \textit{disintegrated} if, for all $A\subset\M$, $\acl(A)=\bigcup\{\acl(a):a\in A\text{ is a singleton}\}$; 
\item $\acl$ is \textit{trivial} if $\acl(A)=A$ for all $A\subset\M$.
\end{enumerate}

We now define axioms of abstract ternary relations on (small subsets of) $\M$. Some axioms have been slightly adjusted from their standard formulations, and incorporate algebraic closure of the small subsets in question. 

\begin{definition}\label{axioms}
Given a ternary relation $\ind$ on $\M$, define the following axioms. 
\begin{enumerate}[$(i)$]
\item \textit{(invariance)} For all $A,B,C$, if $A\ind_C B$ and $\sigma\in\Aut(\M)$ then $\sigma(A)\ind_{\sigma(C)}\sigma(B)$.
\item \textit{(monotonicity)} For all $A,B,C$, if $A\ind_C B$, $A_0\seq A$, and $B_0\seq B$, then $A_0\ind_C B_0$.
\item \textit{(symmetry)} For all $A,B,C$, if $A\ind_C B$ then $B\ind_C A$.
\item \textit{(full transitivity)} For all $A$ and $D\seq C\seq B$, $A\ind_D B$ if and only if $A\ind_C B$ and $A\ind_D C$.
\item \textit{(full existence)} For all $B,C\subset\M$ and tuples $a\in\M$, if $C$ is closed then there is $a'\equiv_C a$ such that $a'\ind_C B$.
\item \textit{(stationarity)} For all closed $C\subset\M$ and closed tuples $a,a',b\in\M$, with $C\seq a\cap b$, if $a\ind_C b$, $a'\ind_C b$, and $a'\equiv_C a$, then $ab\equiv_C a'b$.
\item \textit{(freedom)} For all $A,B,C,D$, if $A\ind_C B$ and $C\cap AB\seq D\seq C$, then $A\ind_D B$.
\item \textit{(closure)} For all closed $A,B,C$, if $C\seq A\cap B$ and $A\ind_C B$ then $AB$ is closed.
\end{enumerate}
\end{definition}

There is a significant body of literature concerning axioms of ternary notions of independence. An excellent introduction  can be found in \cite{Adgeo}.  The choice of axioms in Definition \ref{axioms} also borrows heavily from Tent and Ziegler's work with \textit{stationary independence relations} \cite{TeZiSIR}, and so we give the following adaptation of their definition to the present context.

\begin{definition}\label{def:SIR}
A ternary relation $\ind$ is a  \textbf{stationary independence relation} for $T$ if it satisfies invariance, monotonicity, symmetry, full transitivity, full existence, and stationarity.
\end{definition}

Several comments are warranted at this point. First, Tent and Ziegler's definition in \cite{TeZiSIR} is formulated for finite subsets of a countable structure, and does not include any closure assumptions in the full existence or stationarity axioms. Moreover, the clause ``$C\seq a\cap b$" is not present in their formulation of stationarity. The main examples in \cite{TeZiSIR} have trivial algebraic closure and, in such cases, one may show that the two notions of a stationary independence relation are the same. In \cite{EGT}, Evans, Ghadernezhad, and Tent also consider axioms of ternary relations, which have been relativized to the lattice of algebraically closed sets. 

The clause ``$C\seq a\cap b$" in the stationarity axiom will be necessary in the subsequent work. On the other hand, one may easily show that, if  $\ind$ is a ternary relation satisfying monotonicity, then the full existence axiom is equivalent to the version obtained by adding ``$C\seq a\cap B$" to the assumptions. We will tacitly use this observation when discussing examples in the next section.

Finally, we point out that Tent and Ziegler \cite{TeZiSIR} do formulate the freedom axiom (although they do not give it a name). This axiom is also very close to Hrushovski's notion of CM-triviality \cite{HrSMS}. 

We now define the central notion of this paper.

\begin{definition}\label{def:FAT}
A ternary relation is a \textbf{free amalgamation relation} if it satisfies invariance, monotonicity, symmetry, full transitivity, full existence, stationarity, freedom, and closure. $T$ is a \textbf{free amalgamation theory} if it has a free amalgamation relation.
\end{definition}

The main results of this paper concern properties of free amalgamation theories. The reader will notice that some results do not, in and of themselves, require all parts of the previous definition. Therefore, to obtain the conclusion of a particular result, it may not be necessary for $T$ to have a ternary relation satisfying every ingredient of Definition \ref{def:FAT}.

\section{Examples}\label{sec:ex}

In order to state the motivating examples of free amalgamation theories, we must first define the notion of free amalgamation of relational structures, which gives rise to the canonical example of a ternary relation satisfying the freedom axiom. Given a relational language $\cL$ and $\cL$-structures $A,B,C$, we write $A\cong_C B$ if there is an $\cL$-isomorphism from $AC$ to $BC$, which fixes $C$ pointwise.

\begin{definition}\label{def:FA}
Assume $\cL$ is relational. Given an $\cL$-structure $\cM$ and $A,B,C\seq \cM$, we set $A\ind^{fa}_C B$ (in $\cM$) if $A\cap B\seq C$ and, for all $R\in\cL$ and $a\in ABC$ (of length the arity of $R$), if $R(a)$ holds then $a\in AC$ or $a\in BC$. 
\end{definition}

To mitigate possible confusion, we emphasize that we are now using the phrase ``free amalgamation" in two different ways. In particular, when we say $\ind$ is a ``free amalgamation relation", or $T$ is a ``free amalgamation theory", we mean with respect to the definition involving abstract axioms of ternary relations. When considering structures in a relational language, we will use ``free amalgamation of relational structures" (or ``free amalgamation of $\cL$-structures") when referring to the notion of freely amalgamating such structures as in the previous definition. 

\begin{example}\label{ex}$~$
\begin{enumerate}
\item\label{ex1}\textbf{\Fraisse\ limits with free amalgamation.} Let $\cL$ be a finite relational language and let $\cM$ be the \Fraisse\ limit of a \Fraisse\ class $\cK$ of finite $\cL$-structures, which is \emph{closed under free amalgamation of $\cL$-structures}, i.e., for all $A,B,C\in\cK$, with $C\seq A\cap B$, there is $D=A'B'\in\cK$ such that $A'\cong_C A$, $B'\cong_C B$, and $A'\ind^{fa}_C B'$ (in $D$). In this case, $\Th(\cM)$ is $\aleph_0$-categorical and $\acl$ is trivial (see \cite[Chapter 7]{Hobook}). We give a few examples. 
\begin{enumerate}[$(i)$]
\item If $\cK$ is the class of graphs, then $\cM$ is the \emph{random graph} or \emph{Rado graph}.
\item Given fixed $n>r\geq 2$, let $\cK$ be the class of $K^r_n$-free $r$-hypergraphs, where $K^r_n$ is the complete $r$-hypergraph on $n$ vertices, considered in the $r$-hypergraph language containing an $r$-ary relation symbol. Then $\cM$ is the \emph{generic $K^r_n$-free $r$-hypergraph}. When $r=2$, we also refer to $\cM$ as the \emph{generic $K_n$-free graph} or \emph{Henson graph}.
\item Let $\cK$ be the class of finite metric spaces, with distances in $\{0,1,2,3\}$, in the language $\cL=\{d_1(x,y),d_3(x,y)\}$ where, for $r\in\{1,3\}$, $d_r(x,y)$ is a binary relation interpreted as $d(x,y)=r$. If $A,B,C,D\in\cK$, with $C\seq A\cap B$ and $AB\seq D$, then $A\ind^{fa}_C B$ (in $D$) if and only if $d(a,b)=2$ for all $a\in A\backslash C$ and $b\in B\backslash C$. By the triangle inequality, $\cK$ is closed under free amalgamation of $\cL$-structures. The \Fraisse\ limit $\cM$ is the \emph{Urysohn space with spectrum $\{0,1,2,3\}$}. This structure is also called the \emph{free third root of the complete graph} by Casanovas and Wagner in \cite{CaWa}. Note that $\ind^{fa}$ is \emph{not} the usual free amalgamation of metric spaces, which is the stationary independence relation used by Tent and Ziegler \cite{TeZiSIR} in their analysis of the rational Urysohn space. In general, free amalgamation of metric spaces fails the freedom axiom.
\end{enumerate}

\item\label{ex2}\textbf{Generic $(K_n+K_3)$-free graph.} Fix $n\geq 3$ and let $K_n+K_3$ be the graph obtained by freely amalgamating $K_n$ and $K_3$ over a single vertex. In \cite{CSS}, Cherlin, Shelah, and Shi construct the unique countable, universal, existentially closed $(K_n+K_3)$-free graph, which we denote $\cG_n$ ($\cG_3$ was originally constructed by Komj\'{a}th \cite{Kom}). For any $n\geq 3$, $\Th(\cG_n)$ is $\aleph_0$-categorical and $\acl$ is disintegrated (see \cite{CSS}). However, the age of $\cG_n$ is \emph{not} closed under free amalgamation of arbitrary relational structures (e.g. $K_n+K_3$ itself is obtained as the free amalgamation of two $(K_n+K_3)$-free graphs). Accordingly, the age of $\cG_n$ is not a \Fraisse\ class in the graph language. However, it is shown in \cite{PaSOP4} that this class is closed under free amalgamation of relational structures over \emph{algebraically closed} base structures.

\item\label{ex3}\textbf{``Freely disintegrated" Hrushovski constructions.} Let $\cL$ be a finite relational language, and let $\cM_f$ be the Hrushovski generic produced from a class $(\cK_f,\leq)$ of finite structures closed under free amalgamation of strong substructures, where $f$ is a ``good" control function (see \cite{Evans}, \cite{EvWo} for details). In this case, $\Th(\cM_f)$ is $\aleph_0$-categorical, but $\ind^{fa}$ does not necessarily satisfy the closure axiom, and so we must separately impose this assumption. Note that, if $\acl$ is disintegrated and $A,B$ are closed, then $AB$ is closed as well. Therefore, the closure axiom for $\ind^{fa}$ is asserting that $\acl$ is ``freely disintegrated". It will follow from results in Section \ref{sec:simple} that any simple Hrushovski construction satisfying these assumptions is modular, and so this framework is not suitable for the well-known non-modular Hrushovski counterexamples.
\end{enumerate}
\end{example}

We will show that if $\cM$ is one of the countable structures defined in Example \ref{ex}, then $\Th(\cM)$ is a free amalgamation theory. First, we note that in any relational structure, the ternary relation $\ind^{fa}$ always satisfies several of our axioms (most importantly, $\ind^{fa}$ satisfies freedom).  

\begin{proposition}\label{ISF}
Assume $\cL$ is relational and $\cM$ is an $\cL$-structure. Then $\ind^{fa}$ satisfies invariance, monotonicity, symmetry, full transitivity, and freedom (in $\cM$).
\end{proposition}

The proof is straightforward, and left to the reader. With this result, we see that in order to use $\ind^{fa}$ to obtain a free amalgamation relation for the previous examples, the key axioms to verify are \emph{full existence}, \emph{stationarity}, and \emph{closure}.

\begin{proposition}\label{prop:ex}
Suppose $\cM$ is one of the countable structures described in Example \ref{ex}. Then $\ind^{fa}$ is a free amalgamation relation for $\Th(\cM)$.
\end{proposition}
\begin{proof}
We need to verify that $\ind^{fa}$ satisfies existence, stationarity, and closure. By $\aleph_0$-categoricity, it suffices to work with finite subsets of $\cM$. In each example, the closure axiom is either by assumption or follows from the fact that $\acl$ is disintegrated, and so the the union of two closed sets is closed. The existence axiom for \ref{ex}.\ref{ex1} and \ref{ex}.\ref{ex3} is by assumption, and is shown for \ref{ex}.\ref{ex2} in \cite{PaSOP4}. For stationarity, fix finite, closed $C\subset\cM$ and $a,a',b\in\cM$ such that $C\seq a\cap b$, $a\ind^{fa}_C b$, $a'\ind^{fa}_C b$ and $a\equiv_C a'$. We want to show $a'b\equiv_C ab$. From $a'\equiv_C a$, $a\ind^{fa}_C b$, and $a'\ind^{fa}_C b$, it follows that $ab\cong_C a'b$. Moreover, $ab$ and $a'b$ are closed since $\ind^{fa}$ satisfies closure. Therefore, $a'b\equiv_C ab$ follows from the fact that any $\cL$-isomorphism between finite closed subsets of $\cM$ extends to an automorphism of $\cM$ (see \cite{Hobook}, \cite{PaSOP4}, \cite{EvWo} for, respectively, \ref{ex}.\ref{ex1}, \ref{ex}.\ref{ex2}, \ref{ex}.\ref{ex3}). 
\end{proof}

The interested reader should consult the sources mentioned in the previous proof to find explicit descriptions of algebraic closure in the three families of examples. We also remark that the assumption of a finite language in Examples \ref{ex}.\ref{ex1} and \ref{ex}.\ref{ex3} is there to ensure $\aleph_0$-categoricity and the appropriate level of quantifier elimination. This assumption can be weakened slightly to encompass countable relational languages with only finitely many relations of any given arity, provided that we restrict to structures in which the interpretation of any relation is irreflexive.

In the proof of Proposition \ref{prop:ex}, we used the closure axiom to prove stationarity. We will not explicitly use the closure axiom again until Section \ref{sec:simple}.

\section{NSOP\textsubscript{4}}

In this section, we show that free amalgamation theories form a subclass of first-order theories without the strict order property. In fact, we prove these theories are $\NSOP_4$. This has been shown for each of the examples in the previous section by collective work of several authors including Shelah  \cite{Sh500}, Hrushovski \cite{Udibook}, Evans \& Wong \cite{EvWo}, Patel \cite{PaSOP4}, and joint work with Terry \cite{CoTe}. The most general argument in this direction can be found in unpublished work of Patel \cite{PaSOP4}, which proves $\NSOP_4$ for Example \ref{ex}.\ref{ex1} and Example \ref{ex}.\ref{ex2}. Our argument for $\NSOP_4$, while slightly simpler and more general, is very close to Patel's work. 

We continue to fix a first-order theory $T$ and a monster model $\M$. We begin with the definition of $\SOP_n$. 

\begin{definition}\label{SOPn}
Given $n\geq 3$, $T$ has the \textbf{$n$-strong order property}, $\SOP_n$, if there is an indiscernible sequence $(a_i)_{i<\omega}$ such that, if $p(x,y)=\tp(a_0,a_1)$, then
$$
p(x_1,x_2)\cup\ldots\cup p(x_{n-1},x_n)\cup p(x_n,x_1)
$$
is inconsistent. We say $T$ is $\NSOP_n$ if it does not have the $n$-strong order property.
\end{definition}

\begin{remark}
These properties were originally defined in \cite{Sh500} to enrich the classification of unstable theories. It is fairly straightforward to show that if $T$ has the strict order property, then it has $\SOP_n$ for all $n\geq 3$. Given $n\geq 3$, if $T$ has $\SOP_{n+1}$ then it has $\SOP_n$. Moreover, if $T$ has $\SOP_3$ then it is unstable. Indeed, if one were to interpret Definition \ref{SOPn} with $n=2$ then, as a property of $T$, the result would be equivalent to the order property.\footnote{We caution that $\SOP_2$ is \emph{not} defined this way, but rather as a variant of the tree property.} See \cite{KiKi}, \cite{Sh500}.
\end{remark}

We now return to free amalgamation theories. The following easy observation will be very useful.

\begin{lemma}\label{switch}
Suppose $\ind$ satisfies invariance, symmetry, and stationarity. Suppose $C\subset\M$ is closed and $a,b\in\M$ are closed tuples, with $C\seq a\cap b$, $a\equiv_C b$, and $a\ind_C b$. Then $ab\equiv_C ba$.
\end{lemma}
\begin{proof}
Since $a\equiv_C b$, there is $a'\in\M$ such that $ab\equiv_C ba'$. By invariance and symmetry we have $a'\ind_C b$. Then $ba'\equiv_C ba$ by stationarity.
\end{proof}

\begin{theorem}\label{thm:NSOP4}
If $T$ is a free amalgamation theory then $T$ is $\NSOP_4$.
\end{theorem}
\begin{proof}
Fix an indiscernible sequence $(a_i)_{i<\omega}$ and let $p(x,y)=\tp(a_0,a_1)$. We want to show 
\begin{equation*}
\text{$p(x_1,x_2)\cup p(x_2,x_3)\cup p(x_3,x_4)\cup p(x_4,x_1)$ is consistent.}\tag{$\dagger$}
\end{equation*}

Let $a'_0$ be such that $\acl(a_0)=a_0a'_0$. Given $i<\omega$ there is $\sigma_i\in\Aut(\M)$ such that $\sigma_i(a_0)=a_i$. Let $a'_i=\sigma_i(a'_0)$. Then $b_i:=\acl(a_i)=a_ia'_i=\sigma_i(a_0a'_0)$. Note that, for any $i<j<\omega$, if $q_{i,j}(xx',yy')=\tp(b_i,b_j)$ then $p(x,y)\seq q_{i,j}(xx',yy')$. Therefore, we may replace $(b_i)_{i<\omega}$ by an indiscernible realization of its $\EM$-type and maintain this feature (see, e.g., \cite[Lemma 7.1.1]{TeZi}\footnote{Recall that the $\EM$-type of a sequence $(b_i)_{i<\omega}$, over parameters $A$, is the collection of formulas $\vphi(x_1,\ldots,x_n)$ over $A$, for any $n<\omega$, such that $\vphi(b_{i_1},\ldots,b_{i_n})$ holds for all $i_1<\ldots<i_n<\omega$.}).

Set $q(xx',yy')=q_{0,1}(xx',yy')$. To show $(\dagger)$, we set $z_i=x_ix_i'$ and show
\begin{equation*}
\text{$q(z_1,z_2)\cup q(z_2,z_3)\cup q(z_3,z_4)\cup q(z_4,z_1)$ is consistent.}\tag{$\dagger\dagger$}
\end{equation*}

Let $c\sqsubset b_0$ be the common intersection of $(b_i)_{i<\omega}$, which is closed. Let $\ind$ be a free amalgamation relation for $T$. By full existence there is $b^*_0\equiv_{b_1}b_0$ such that $b^*_0\ind_{b_1}b_2$. Then $b^*_0\cap b_1=c=b_1\cap b_2$, and so $b^*_0\ind_c b_2$ by freedom. Moreover, $b^*_0\equiv_c b_0\equiv_c b_2$, and so $b^*_0b_2\equiv_cb_2b^*_0$ by Lemma \ref{switch}. Let $b^*_1$ be such that $b^*_0b_2b_1\equiv_c b_2b^*_0b^*_1$. We have:
\begin{enumerate}[$(i)$]
\item $b^*_0b_1\equiv b_0b_1$, and so $q(b^*_0,b_1)$,
\item $q(b_1,b_2)$,
\item $b_2b^*_1\equiv b^*_0b_1$, and so $q(b_2,b^*_1)$,
\item $b^*_1b^*_0\equiv b_1b_2$, and so $q(b^*_1,b^*_0)$.
\end{enumerate}
This proves $(\dagger\dagger)$, as desired.
\end{proof}

Note that $\NSOP_4$ is optimal, as many examples in Section \ref{sec:ex} have $\SOP_3$ (e.g. the generic $K_n$-free graph). Moreover, the freedom axiom is necessary to conclude $\NSOP_4$. For example, the theory of the rational Urysohn space has a stationary independence relation satisfying closure (see \cite{TeZiSIR}), but is $\SOP_n$ for all $n\geq 3$ (see \cite{CoDM2}, \cite{CoTe}). We also observe that, in the proof of $\NSOP_4$, algebraic closure could be replaced by any invariant closure operator. 

\begin{example}
In \cite{SUgroups}, Shelah and Usvyatsov consider groups as a universal class. Using amalgamated free products, they prove that if $\G$ is a sufficiently large universal group, then $\G$ is $\NSOP_4$ with respect to quantifier-free types. In particular, given $A,B,C\subset\G$, set $A\ind_C B$ if $\langle ABC\rangle$ is isomorphic to $\langle AC\rangle\ast_{\langle C\rangle}\langle BC\rangle$ via the natural map. Then $\ind$ satisfies all axioms of a free amalgamation relation except closure (where ``closed" sets are subgroups and, in the stationarity and full existence axioms, elementary equivalence is replaced by group isomorphism). Altogether, the proof of Theorem \ref{thm:NSOP4} recovers this result in \cite{SUgroups}.
\end{example}

\section{Imaginaries and Hyperimaginaries}\label{sec:WEI}

The main result of this section is that any free amalgamation theory has elimination of hyperimaginaries and weak elimination of imaginaries. We first recall basic notation and definitions (see also \cite{Cabook}, \cite{Kibook}). Given a $0$-type-definable equivalence relation $E(x,y)$ and a tuple $a\in\M$, with $\ell(a)=\ell(x)$, $a_E$ denotes the \emph{hyperimaginary} determined by $[a]_E$ (the equivalence class of $a$ modulo $E$). If $E(x,y)$ is $0$-definable and $\ell(x)$ is finite, then $a_E$ is an \emph{imaginary}. Given $A\seq\M^{\heq}$ and $e\in\M^{\heq}$, we let $\cO(e/A)$ denote the orbit of $e$ under $\Aut(\M/A)$. Then $\dcl^{\heq}(A)=\{e\in\M^{\heq}:\cO(e/A)=\{e\}\}$ and $\bdd(A)=\{e\in\M^{\heq}:|\cO(e/A)|<\kappa\}$ (where $\kappa$ is the saturation cardinal of $\M$). For $A\subset\M^{\eq}$, let $\edcl(A)=\dcl^{\heq}(A)\cap\M^{\eq}$ and $\eacl(A)=\bdd(A)\cap\M^{\eq}$. 

A theory $T$ has \emph{elimination of hyperimaginaries} if every $e\in\M^{\heq}$ is interdefinable with a sequence in $\M^{\eq}$. Given $e\in\M^{\eq}$, a \emph{geometric canonical parameter} for $e$ is a finite tuple $c\in\M$ such that $c\in\eacl(e)$  and $e\in\eacl(c)$. If $c\in\eacl(e)$ and $e\in\edcl(c)$ then $c$ is a \emph{weak canonical parameter} for $e$. If $c\in\edcl(e)$ and $e\in\edcl(c)$ then $c$ is a \emph{canonical parameter} for $e$. $T$ has (\emph{geometric}, \emph{weak}) \emph{elimination of imaginaries} if every imaginary has a (geometric, weak) canonical parameter.

\begin{definition}
Suppose $E(x,y)$ is a $0$-type-definable equivalence relation on $\M^I$.
\begin{enumerate}
\item  A sequence $(a_i)_{i<\omega}$ in $\M^I$ is \textbf{$E$-related} if $E(a_i,a_j)$ holds for all $i,j<\omega$.
\item Given $a\in\M^I$, define $\Sigma(a,E)$ to be the set of subtuples $c\sqsubset a$ such that there is an $E$-related indiscernible sequence $(a_i)_{i<\omega}$ in $\M^I$, with common intersection $c$ and $a_0=a$. 
\end{enumerate}
\end{definition}

\begin{lemma}
Suppose $E(x,y)$ is a $0$-type-definable equivalence relation on $\M^I$ and $a\in\M^I$. Then $\Sigma(a,E)$ contains a minimal element under inclusion of tuples.
\end{lemma}
\begin{proof}
We use Zorn's Lemma. Note that $a\in\Sigma(a,E)$ (witnessed by the constant sequence $a_i=a$ for all $i<\omega$), and so $\Sigma(a,E)$ is nonempty. Suppose $\lambda$ is an ordinal and $(c^i)_{i<\lambda}$ is a sequence of elements of $\Sigma(a,E)$ such that $i<j$ implies $c^j\sqsubset c^i$. Let $K_i\seq I$ be the domain of $c^i$, and note that $i<j$ implies $K_j\seq K_i$. Set $K=\bigcap_{i<\lambda}K_i$ and $c=a_K$. Then $c\sqsubset c^i$ for all $i<\lambda$, and we show $c\in\Sigma(a,E)$. Consider variables $(x^i)_{i<\omega}$, where $\ell(x^i)=I$, and define the type
$$
\Delta=P\cup\{x^i_K=c:i<\omega\}\cup\{x^i_k\neq x^j_k:i<j<\omega,~k\in I\backslash K\},
$$
where $P$ expresses that $(x^i)_{i<\omega}$ is an indiscernible $E$-related sequence with $x^0=a$. A finite subset of $\Delta$ is contained in a type of the form 
$$
\Delta_0=P\cup\{x^i_K=c:i<\omega\}\cup\{x^i_k\neq x^j_k:i<j<\omega,~k\in I_0\},
$$
where $I_0$ is a finite subset of $I\backslash K$. Given $\Delta_0$, fix $t<\lambda$ such that $I_0\seq I\backslash\bigcap_{i<t}K_i$. By assumption, there is an $E$-related indiscernible sequence $(a^i)_{i<\omega}$, with common intersection $c^t$ and $a^0=a$. This sequence realizes $\Delta_0$. By compactness, $\Delta$ is consistent, and so $c\in\Sigma(a,E)$. 
\end{proof}

\begin{definition}
Suppose $E(x,y)$ is a $0$-type-definable equivalence relation on $\M^I$, and $a\in \M^I$. An \textbf{indiscernible parameter for $a_E$} is a minimal element (under $\sqsubset$) of $\Sigma(a,E)$.
\end{definition}

\begin{lemma}\label{PHPlem}
Suppose $E(x,y)$ is a $0$-type-definable equivalence relation on $\M^I$. Fix $a\in\M^I$ and let $c\sqsubset a$ be an indiscernible parameter for $a_E$. Then $c\in\bdd(a_E)$. 
\end{lemma}
\begin{proof}
We may clearly assume $c$ is nonempty. Let $I_0\seq I$ be the domain of $c$. Let $(a^l)_{l<\omega}$ be an $E$-related indiscernible sequence, with common intersection $c$, such that $a^0=a$.

Suppose, toward a contradiction, that $c\not\in\bdd(a_E)$, and so we may find a nonconstant sequence $(c^l)_{l<\lambda}$ in $\cO(c/a_E)$, with $c^0=c$ and $\lambda$ arbitrarily large. Choosing $\lambda$ large enough, we may assume $(c^l)_{l<\lambda}$ is indiscernible. For later purposes, we also want $\lambda\geq(2^{|I|+\aleph_0})^+$. By compactness, we may stretch $(a^l)_{l<\omega}$ so that it is indexed $(a^l)_{l<\lambda}$ (and still indiscernible, $E$-related, with common intersection $c$).
 
Let $I_1\seq I_0$ be the domain of the common intersection $d$ of $(c^l)_{l<\lambda}$. Since $(c^l)_{l<\lambda}$ is not a constant sequence, we must have $I_1\neq I_0$, and so $d$ is a proper subtuple of $c$. We will build an indiscernible $E$-related sequence $(b^m)_{m<\omega}$, with $b^0=a$, such that the common intersection of $(b^m)_{m<\omega}$ is a subtuple of $d$. This will contradict the minimality of $c$. 

Given $J\seq I$, define a binary relation $D_J$ on $\M^I$ such that $D_J(b,b')$ holds if and only if, for all $s,t\in I$, if $b_s=b'_t$ then $s,t\in J$. Given $l<\kappa$, fix $\sigma_l\in\Aut(\M/a_E)$ such that $\sigma_l(c)=c^l$ (we assume $\sigma_0$ is the identity). We inductively construct a sequence $(b^m)_{m<\omega}$ and an injective function $f:\omega\func\lambda$ such that
\begin{enumerate}[$(i)$]
\item $b^0=a$;
\item for all $m<\omega$ there is $r<\lambda$ such that $b^m=\sigma_{f(m)}(a^r)$;
\item $D_{I_0}(b^l,b^m)$ holds for all $l<m<\omega$.
\end{enumerate}

We first argue why this construction finishes the proof. Suppose we have $(b^m)_{m<\omega}$ as described. To show $(b^m)_{m<\omega}$ is $E$-related, we fix $m<\omega$ and show $E(b^m,a)$ holds. By $(ii)$, we have some $r<\lambda$ such that $b^m=\sigma_{f(m)}(a^r)$. Note that $E(a^r,a)$ holds since $(a^l)_{l<\lambda}$ is $E$-related. Since $\sigma_{f(m)}\in\Aut(\M/a_E)$, we then have $E(\sigma_{f(m)}(a^r),a)$. Next, we fix $l<m<\omega$ and show that $D_{I_1}(b^l,b^m)$ holds. Indeed, if $s,t\in I$ and $b^l_s=b^m_t$, then we have $s,t\in I_0$ by $(iii)$. It follows from $(ii)$ that $b^l_s=c^{f(l)}_s$ and $b^m_t=c^{f(m)}_t$ (recall $c=a^r_{I_0}$ for all $r<\lambda$). Then $f(l)\neq f(m)$ implies $s,t\in I_1$. Finally, if we replace $(b^m)_{m<\omega}$ by an indiscernible realization of its $\EM$-type, and conjugate to maintain $b^0=a$, then $(b^m)_{m<\omega}$ is the desired $E$-related indiscernible sequence, whose common intersection is a subtuple of $d$.

We now proceed with the construction of $(b^m)_{m<\omega}$. Let $b^0=a$ and $f(0)=0$. For the induction hypothesis, fix $n>0$, and suppose we have constructed $(b^m)_{m<n}$ and $f:n\func\lambda$ satisfying properties $(i)$, $(ii)$, and $(iii)$ above (relativized to $n$).

\noindent\textit{Claim}: There are $\mu,r<\lambda$ such that, for all $m<n$,  $\mu\neq f(m)$ and $D_{I_0}(b^m,\sigma_\mu(a^r))$.

Note that, given the claim, if we set $f(n)=\mu$ and $b^n=\sigma_\mu(a^r)$, then $b^n$ and $f:n+1\func\lambda$ are as desired. Therefore the claim finishes the inductive step in the construction of $(b^m)_{m<\omega}$.

\noindent\textit{Proof of the claim}: Suppose the claim fails. Then, for all $\mu\in\lambda\backslash\Im(f)$ and $r<\lambda$, there are $m<n$ and $s,t\in I$ such that $\{s,t\}\not\seq I_0$ and $b^m_s=\sigma_\mu(a^r_t)$. We first find an integer $m<n$, indices $s,t\in I$ with $\{s,t\}\not\seq I_0$, and $2$-element sets $\Delta_1,\Delta_2\seq\lambda\backslash\Im(f)$ such that $b^m_s=\sigma_\mu(a^r_t)$ for all $(\mu,r)\in\Delta_1\times\Delta_2$. To do this, set
$$
X=\{(\mu,r)\in(\lambda\backslash\Im(f))^2:r<\mu\}\mand Y=\{(m,s,t):m<n,~s,t\in I\},
$$
and  consider the map $\tau:X\func Y$ obtained above from the assumption that the claim fails. Let $\theta=|I|+\aleph_0$. We have $\lambda\geq (2^\theta)^+$ by assumption, and so $\lambda\rightarrow (\theta^+)^2_\theta$ by the Erd\H{o}s-Rado Theorem (see, e.g., \cite[Theorem C.3.2]{TeZi}). Applied to the map $\tau$, we obtain an infinite\footnote{Erd\H{o}s-Rado ensures $|\Delta|\geq\theta^+$; however we only need $|\Delta|\geq 4$ for the proof.} set $\Delta\seq\lambda\backslash\Im(f)$, an integer $m<n$, and $s,t\in I$ such that $\{s,t\}\not\seq I_0$ and $b^m_s=\sigma_\mu(a^r_t)$ for all $(\mu,r)\in X\cap\Delta^2$. Now let $\Delta_1,\Delta_2\seq\Delta$ be disjoint $2$-element sets with $\max\Delta_2<\min\Delta_1$.

Now fix some $\mu\in\Delta_1$ and distinct $r,r'\in \Delta_2$. Then $\sigma_\mu(a^r_t)=b^m_s=\sigma_\mu(a^{r'}_t)$, and so $a^r_t=a^{r'}_t$. In particular, we must have $t\in I_0$, and so $s\not\in I_0$. Moreover, for any $r<\lambda$, we have $a^r_t=a_t$. Therefore, given $\mu\in\Delta_1$, we have $b^m_s=\sigma_\mu(a_t)$.

Fix distinct $\mu,\mu'\in\Delta_1$. Since $t\in I_0$, we altogether have $c^\mu_t=\sigma_\mu(a_t)=b^m_s=\sigma_{\mu'}(a_t)=c^{\mu'}_t$, and so $t\in I_1$. Therefore $a_t=c_t\in d$ and so $\sigma_\mu(a_t)=a_t$ for all $\mu\in\Delta_1$. In particular, we have $b^m_s=a_t=b^0_t$, and so $D_{I_0}(b^0,b^m)$ fails. If $m>0$ then we obtain a contradiction to property $(iii)$ in the induction hypothesis. Therefore $m=0$, and so $a_s=b^0_s=a_t$. By indiscernibility of $(a^l)_{l<\lambda}$, and since $t\in I_0$, we have $a^1_s=a^1_t=a_t=a_s$, which contradicts $s\not\in I_0$.
\end{proof}

\begin{lemma}\label{dcllem}
Suppose $\ind$ is a ternary relation on $\M$ satisfying invariance, monotonicity, full existence, stationarity, and freedom. Let $E(x,y)$ be a $0$-type-definable equivalence relation on $\M^I$, and suppose $a\in\M^I$ is closed. Then $a_E\in\dcl^{\heq}(c)$ for any $c\in \Sigma(a,E)$.
\end{lemma}
\begin{proof}
Fix $c\in \Sigma(a,E)$, and let $(a_i)_{i<\omega}$ be an $E$-related indiscernible sequence, with common intersection $c$, such that $a_0=a$. Note that $(a_i)_{i<\omega}$ is $c$-indiscernible and $c$ is closed. To show $a_E\in\dcl^{\heq}(c)$, we fix $\sigma\in\Aut(\M/c)$ and show $E(a,\sigma(a))$.    

By full existence, there is $b\equiv_{a_1}a$ such that $b\ind_{a_1}a_2$. Then $ba_1\equiv_c aa_1$, and so $b\cap a_1=c=a_1\cap a_2$. Therefore $b\ind_c a_2$ by freedom. Note also that, since $E(a,a_1)$ holds, we have $E(b,a_1)$, and so, combined with $E(a_1,a_2)$, we obtain $E(b,a_2)$. 

By full existence, there is $b_*\equiv_c a$ such that $b_*\ind_c a\sigma(a)$. Since $a\equiv_c a_2$, we may fix $b_1$ such that $b_*a\equiv_c b_1a_2$. Then $b_1\equiv_c b_*\equiv_c a\equiv_c b$ and, by monotonicity and invariance, $b_1\ind_c a_2$. Therefore, by stationarity, $ba_2\equiv_c b_1a_2 \equiv_c b_*a$. In particular, we have $E(b_*,a)$. 

Similarly, since $\sigma(a)\equiv_c a \equiv_c a_2$, we may fix $b_2$ such that $b_*\sigma(a)\equiv_c b_2a_2$. Then $b_2\equiv_c b_* \equiv_c b$ and, by monotonicity and invariance, $b_2 \ind_c a_2$. Therefore, by stationarity, $ba_2 \equiv_c b_2a_2 \equiv_c b_*\sigma(a)$. In particular, we have $E(b_*,\sigma(a))$. Altogether, we have $E(b_*,a)$ and $E(b_*,\sigma(a))$, and so $E(a,\sigma(a))$ holds, as desired.
\end{proof}

\begin{theorem}\label{thm:WEI}
If $T$ is a free amalgamation theory then $T$ has elimination of hyperimaginaries and weak elimination of imaginaries.
\end{theorem}
\begin{proof}
Both results rely on the following claim.

\noindent\textit{Claim}: Suppose $E(x,y)$ is a $0$-type-definable equivalence relation on $\M^I$, and $a\in\M^I$. Then there is a real tuple $c\in\M$ such that $c\in\bdd(a_E)$ and $a_E\in\dcl^{\heq}(c)$.

\noindent\textit{Proof}: Let $a_*$ be a tuple, with domain $I_*$, such that $I\cap I_*=\emptyset$ and $aa_*=\acl(a)$. Consider the equivalence relation $E_*$ on $\M^{I+I^*}$ given by $E_*(x_Ix_{I_*},y_Iy_{I_*})$ if and only if $E(x_I,y_I)$. Then $E_*$ is $0$-type-definable, and so, if $c$ is an indiscernible parameter for $e:=(aa_*)_{E_*}$, then $c\in\bdd(e)$ and $e\in\dcl^{\heq}(c)$ by Lemmas \ref{PHPlem} and \ref{dcllem}.

Note that $\Aut(\M/a_E)=\Aut(\M/e)$, and so $\cO(c/a_E)=\cO(c/e)$, which is bounded by assumption. Therefore $c\in\bdd(a_E)$. Moreover, if $\sigma\in\Aut(\M/c)$ then $\sigma(e)=e$, which means $E_*(aa_*,\sigma(aa_*))$ holds, and so $E(a,\sigma(a))$ holds by definition. Therefore $\sigma(a_E)=a_E$, and so $a_E\in\dcl^{\heq}(c)$. \claim

By the claim, and \cite[Lemma 18.6]{Cabook}, we immediately obtain elimination of hyperimaginaries. For weak elimination of imaginaries, fix a $0$-definable equivalence relation on $\M^n$, and $a\in\M^n$. Let $c\in\M$ be as in the claim. Then $a_E\in\edcl(c)$ (with $c$ considered as a small subset of $\M$). Let $c_0$ be a finite subtuple of $c$ such that $a_E\in\edcl(c_0)$. Then $c\in\bdd(a_E)$ implies $c_0\in\bdd(a_E)\cap\M\seq\eacl(a_E)$, and so $c_0$ is a weak canonical parameter for $a_E$.
\end{proof}

We  remark that Theorem \ref{thm:WEI} cannot be strengthened to full elimination of imaginaries. Indeed, one often has that finite imaginaries in theories of homogeneous structures, in symmetric relational languages, do not have canonical parameters. For example, this is the case for the random graph, generic $K_n$-free graph, and even just the infinite set. It is also worth observing that the freedom axiom is necessary in Theorem \ref{thm:WEI}. For example, any generic theory of infinitely refining equivalence relations fails weak elimination of imaginaries, but does have a stationary independence relation satisfying closure, namely, nonforking independence. Moreover, the theory of the rational Urysohn space does not eliminate hyperimaginaries (see \cite{CaWa}, \cite{CoDM2}) but, as previously remarked, has a stationary independence relation satisfying closure.

\section{Thorn-Forking and Rosiness}

In this section, we use weak elimination of imaginaries to establish rosiness for many of the examples in Section \ref{sec:ex}. This subject has been previously investigated in some cases. In particular, rosiness for the random graph and generic $K^r_n$-free hypergraphs (for $r>2$) follows from the fact that these theories are simple. Other examples are known to be rosy due to previous proofs of weak elimination of imaginaries. In particular, weak elimination of imaginaries is shown for $\cU_3$ by Casanovas and Wagner \cite{CaWa}, and for the Hrushovski generics $\cM_f$ by Wong \cite{Wothesis}.

On the other hand, rosiness for the general class of \Fraisse\ limits in Example \ref{ex}.\ref{ex1}, does not appear in previous literature. This includes even the specific case of the generic $K_n$-free graphs. Rosiness for the generic $(K_n+K_3)$-free graphs of Example \ref{ex}.\ref{ex2} is also a new result.

We first state the definition of thorn-forking, which follows  \cite{Adgeo}. 

\begin{definition}
Suppose $T$ is a complete theory and $\M$ is a monster model of $T$.
\begin{enumerate}
\item A ternary relation $\ind$ satisfies \textbf{local character} if, for all $A\subset\M$, there is a cardinal $\kappa(A)$ such that, for all $B\subset\M$, there is $C\seq B$ such that $|C|<\kappa(A)$ and $A\ind_C B$.
\item Define \textbf{algebraic independence} $\ind^a$:
$$
\textstyle A\ind^a_C B\miff \acl(AC)\cap\acl(BC)=\acl(C).
$$
\item Define \textbf{$M$-independence} $\ind^M$ by ``forcing base monotonicity" on $\ind^a$:
$$
\textstyle A\ind^M_C B\miff A\ind^a_D B\text{ for all $C\seq D\seq\acl(BC)$}.
$$
\item Define \textbf{thorn independence} $\ind^{\thrn}$ by ``forcing extension" on $\ind^M$:
$$
\textstyle A\ind^\thrn_C B \miff \text{for all $\hat{B}\supseteq B$ there is $A'\equiv_{BC} A$ such that $\textstyle A'\ind^M_C \hat{B}$}.
$$
\end{enumerate}
 $T$ is \textbf{rosy} (resp. \textbf{real rosy}) if $\ind^\thrn$ satisfies local character in $\M^{\eq}$ (resp. in $\M$).
\end{definition}

Thorn-forking was developed in order to define the weakest ternary relation satisfying enough basic axioms to be considered a reasonable notion of independence. In many ways, rosy theories are to thorn-forking as simple theories are to forking. However, the region of rosy theories properly extends the simple theories (e.g. o-minimal theories are rosy). See \cite{EaOn}, \cite{OnTH} for further details.

Since rosiness is defined as a property of $T^{\eq}$, an understanding of imaginaries greatly simplifies the work required to determine if a theory is rosy. In particular, if $T$ has weak elimination of imaginaries, then it suffices to check that $T$ is real rosy. This fact is shown explicitly in \cite{EaGo}, implicitly in \cite{EaOn}, and is also an informative exercise in forking calculus.

\begin{fact}\label{fact:rrosy}
Any real rosy theory, with weak elimination of imaginaries, is rosy.
\end{fact}

\begin{corollary}
Any real rosy free amalgamation theory is rosy.
\end{corollary}

In checking real rosiness for our specific examples, the following facts from \cite{Adgeo} will be useful. Recall that a ternary relation $\ind$ satisfies \emph{base monotonicity} if, for all $A,B,C,D\subset\M$, with $D\seq C\seq B$, if $A\ind_D B$ then $A\ind_C B$.

\begin{definition}\label{def:mod}
A theory $T$ is \textbf{modular} if $\ind^a$ satisfies base monotonicity in $\M$.
\end{definition}

Recall that if $\acl$ satisfies Steinitz exchange in $T$, then the resulting dimension function is used to define a notion of ``modularity" for $T$. One may show that, in this case, the two notions are equivalent (see \cite{Adthesis}, \cite{Adgeo} for details).

\begin{proposition}\label{prop:alg}
Suppose $T$ is a complete theory and $\M$ is a monster model of $T$.
\begin{enumerate}[$(a)$]
\item $T$ is modular if and only if, for all algebraically closed sets $A,B,C\subset\M$, with $C\seq B$, we have $\acl(AC)\cap B=\acl((A\cap B)C)$. 
\item Suppose $T$ is modular. Then, for any $A,B\subset\M$, there is $C\seq B$, with $|C|<\max\{|\acl(A)|^+,\aleph_0\}$, such that $A\ind^a_C B$.
\item If $T$ is modular then $\ind^a=\ind^\thrn$ in $\M$; so $\ind^\thrn$ satisfies local character in $\M$.
\item If algebraic closure in $T$ is disintegrated then $T$ is modular.
\end{enumerate}
\end{proposition}
\begin{proof}
Part $(d)$ follows easily from part $(a)$. Parts $(a)$ and $(c)$ can be found in \cite[Proposition 1.5]{Adgeo}, which, moreover, includes a general argument that $\ind^a$ satisfies local character, even without the modularity assumption. However, our formulation of part $(b)$ uses modularity to conclude a stronger bound on the cardinal $\kappa(A)$ in the local character axiom, and so we detail the argument.

Part $(b)$. Let $D=\acl(A)\cap\acl(B)$. For any singleton $d\in D$, we may fix a finite subset $C_d\seq B$ such that $d\in\acl(C_d)$. Let $C=\bigcup_{d\in D} C_d\seq B$. Then $D\seq\acl(C)$ and $|C|<\max\{|\acl(A)|^+,\aleph_0\}$. Let $A'=\acl(A)$, $B'=\acl(B)$ and $C'=\acl(C)$. Then $C'\seq B'$ and so, using part $(a)$, we have
$$
\acl(AC)\cap\acl(BC)\seq \acl(A'C')\cap B'=\acl((A'\cap B')C')=\acl(DC')=\acl(C).
$$
Therefore $A\ind^a_C B$, as desired.
\end{proof}

Altogether, if $T$ is modular with weak elimination of imaginaries, then $T$ is rosy and $\ind^a=\ind^\thrn$ in $\M$ (and also in $\M^{\eq}$; see Lemma \ref{GEI}). Therefore, we have following conclusion.

\begin{corollary}\label{cor:FArosy}
If $T$ is a modular free amalgamation theory then $T$ is rosy. 
\end{corollary}

Recall that $\acl$ is disintegrated in Examples \ref{ex}.\ref{ex1} and \ref{ex}.\ref{ex2}, and so these theories are modular by Proposition \ref{prop:alg}$(d)$. Combined with the fact that $\acl$ is locally finite, we can use Proposition \ref{prop:alg} to conclude that the structures in these examples yield \emph{superrosy} theories (i.e. in $\M^{\eq}$, $\ind^{\thrn}$ satisfies the strengthening of local character obtained by demanding $\kappa(A)=\aleph_0$ for all finite $A$). Note that superrosiness is also a property of $T^{\eq}$ and so, to justify the previous remark, one must verify that Fact \ref{fact:rrosy} still holds when ``rosy" is replaced by ``superrosy". We again leave this to the reader, and instead turn our attention to calculating the $U^\thrn$-rank of these examples.

\begin{definition}
Suppose $T$ is a complete theory and $\M$ is a monster model of $T$. Given $n<\omega$, $U^{\thrn}(T)\geq n$ if there is a singleton $a\in\M$ and subsets $\emptyset=B_0\seq B_1\seq\ldots\seq B_n\subset\M^{\eq}$ such that $a\nind^{\thrn}_{B_i}B_{i+1}$ for all $i<n$.
\end{definition}

Similar to before, if $T$ has weak elimination of imaginaries, then the subsets $B_i$ in the previous definition may be taken from $\M$. We can now calculate the $U^{\thrn}$-rank of the structures in Examples \ref{ex}.\ref{ex1} and \ref{ex}.\ref{ex2}. For Example \ref{ex}.\ref{ex1}, the following observation implies that the $U^{\thrn}$-rank is $1$.

\begin{proposition}\label{prop:triv}
Suppose $T$ is modular with weak elimination of imaginaries. Then $U^\thrn(T)=1$ if and only if algebraic closure in $\M$ satisfies Steinitz exchange.
\end{proposition}
\begin{proof}
The reverse direction is left to the reader, and in fact holds just under the assumption of geometric elimination of imaginaries (see \cite[Theorem 4.12]{EaOn}). For the forward direction, if $\acl$ fails exchange then we may fix some $a,b\in\M$, and $C\subset\M$ such that $b\in\acl(aC)\backslash\acl(C)$ and $a\not\in\acl(bC)$. In other words, $a\nind^a_C b$ and $a\nind^a_{bC} a$. Since $T$ is modular with weak elimination of imaginaries, this gives $U^{\thrn}(a/C)\geq 2$.
\end{proof}

For Example \ref{ex}.\ref{ex2}, we first set some notation (taken from \cite{PaSOP4}). Given $n\geq 3$, let $T_n=\Th(\cG_n)$ denote the theory of the generic $(K_n+K_3)$-free graph. A singleton $a\in\M\models T_n$ is \textit{type I} if it lies on exactly one $K_n$ in $\M$, and on no $K_3$ other than those occurring as subgraphs of this $K_n$. It is easy to see that type I vertices exist in $\M$. For example, consider the graph obtained by freely amalgamating two copies of $K_n$ over $K_{n-1}$. This graph is $(K_n+K_3)$-free and so we may assume it is a subgraph of $\M$. Moreover, the two vertices not on the common $K_{n-1}$ are each type I. One may also show that if $a$ is type I then $\acl(a)$ is precisely the unique $K_n$ on which $a$ lies. The following technical observations follow from the analysis of algebraic closure found in \cite{CSS} or \cite{PaSOP4}.

\begin{lemma}\label{lem:aclbow}
Fix $n\geq 3$ and let $\M\models T_n$. If $a,b\in\M$ are singletons such that $b\in\acl(a)$ and $a\not\in\acl(b)$, then $a$ is type I and $\acl(a)=\acl(b)\cup\{a\}$. Conversely, if $a$ is type I then $\acl(a)\backslash\{a\}$ is nonempty and $a\not\in\acl(b)$ for any $b\in\acl(a)\backslash\{a\}$.
\end{lemma}

\begin{theorem}\label{thm:bowtie}
For all $n\geq 3$, $U^{\thrn}(T_n)=2$.
\end{theorem}
\begin{proof}
We have $U^{\thrn}(T_n)\geq 2$ by Proposition \ref{prop:triv}, Lemma \ref{lem:aclbow}, and the fact that $\acl(\emptyset)=\emptyset$. For the other direction, recall that by weak elimination of imaginaries and modularity, we may work in $\M$ with $\ind^\thrn=\ind^a$. Suppose, toward a contradiction, there is a singleton $a\in\M$ and $B_0\seq B_1\seq B_2\seq B_3\subset\M$ such that $a\nind^a_{B_i}B_{i+1}$ for all $i<3$. Given $i<3$, fix a singleton $b_{i+1}\in(\acl(aB_i)\cap\acl(B_{i+1}))\backslash\acl(B_i)$. Since $\acl$ is disintegrated, we must have $b_{i+1}\in\acl(a)$ for all $i<3$.  

Since $b_1\in\acl(B_1)$ and $b_2\in\acl(a)\backslash\acl(B_1)$, we have $a\not\in\acl(b_1)$. By Lemma \ref{lem:aclbow}, $a$ is type I and $\acl(a)\seq\acl(B_1)\cup\{a\}$. Then $b_2=a$, which contradicts $b_2\in\acl(B_2)$ and $b_3\in\acl(a)\backslash\acl(B_2)$. 
\end{proof}

Finally, for the sake of completeness, we summarize the previously known result that the Hrushovski constructions in Example \ref{ex}.\ref{ex3} are rosy. This argument works in general, and does not require the assumptions we have imposed in order to obtain free amalgamation theories. Let $\M_f$ be a monster model of $\Th(\cM_f)$. Consider the relation $A\ind^{\dim}_C B$ if and only if $A\ind^a_C B$ and, for all finite $a\in A$, $d(a/BC)=d(a/C)$ (see \cite{Evans}, \cite{EvWo}).\footnote{In the literature, the notation for this ternary relation is $\ind^d$. We use $\ind^{\dim}$ to avoid confusion with nondividing. However, if $\Th(\cM_f)$ is simple then it follows from work in \cite{Evans,EvWo} that $\ind^{\dim}$ coincides with nonforking (and thus also nondividing).} By results in \cite{Evans}, $\ind^{\dim}$ satisfies the axioms of a \emph{strict independence relation} (see \cite{Adgeo}), and so, by \cite[Theorem 4.3]{Adgeo}, $\ind^\thrn$ satisfies local character in $\M_f$ (this fact is observed by Wong in \cite{Wothesis}). Using weak elimination of imaginaries (shown in \cite{Wothesis}), it follows that $\Th(\cM_f)$ is rosy. As noted in \cite{Evans}, if the predimension $d$ is discrete then $\ind^{\dim}$ satisfies the strengthening of local character required to conclude that $\Th(\cM_f)$ is superrosy.

\section{Simplicity}\label{sec:simple}

Many free amalgamation theories are known to be much more well-behaved than what we have shown so far, in particular because they are simple (and therefore $\NSOP_3$). For example, this is true for the random graph and generic $K^r_n$-free hypergraphs (with $r>2$). Moreover, simplicity of Hrushovski constructions is a well-studied topic (see \cite{Evans}). On the other hand, the documented examples of non-simple free amalgamation theories all exhibit a gap in complexity, in the sense that they have both $\SOP_3$ and $\TP_2$. In this section, we investigate the persistence of this behavior. 

Fix a complete first-order theory $T$ and a monster model $\M$. We first define $\TP_2$; and then we give a reformulation of $\SOP_3$ resembling \cite[Claim 2.19]{Sh500}.

\begin{definition}\label{def:TP2}\textcolor{blue}{(This definition not quite right; see Section \ref{sec:C1}.)} 
$T$ has the \textbf{tree property of the second kind}, $\TP_2$, if there are tuples $a,b\in\M$, an array $(b^m_n)_{m,n<\omega}$ in $\M$, and an integer $k<\omega$ such that
\begin{enumerate}[$(i)$]
\item for all $m<\omega$ and $n_1<\ldots<n_k<\omega$, there does not exist a tuple $a_*$ such that $a_*b^m_{n_i}\equiv ab$ for all $1\leq i\leq k$;
\item for all $\sigma:\omega\func\omega$ there is a tuple $a_*$ such that $a_*b^m_{\sigma(m)}\equiv ab$ for all $m<\omega$.
\end{enumerate}
$T$ is $\NTP_2$ if is does not have $\TP_2$.
\end{definition}

\begin{proposition}\label{SOP3} \textcolor{blue}{(This proposition is not quite right; see Section \ref{sec:C1}.)}
$T$ has $\SOP_3$ if and only if there are sequences $(a_i)_{i<\omega}$, $(b_i)_{i<\omega}$ and types $p(x,y)$, $q(x,y)$, with $\ell(x)=\ell(a_0)$ and $\ell(y)=\ell(b_0)$, such that
\begin{enumerate}[$(i)$]
\item $p(a_i,b_j)$ for all $i<j$ and $q(a_i,b_j)$ for all $i\geq j$;
\item for all $i<j$, $p(x,b_i)\cup q(x,b_j)$ is inconsistent.
\end{enumerate}
\end{proposition}
\begin{proof}
We prove the reverse implication (which is the only direction we will use), and leave the forward implication to the reader. Suppose we have $(a_i)_{i<\omega}$, $(b_i)_{i<\omega}$, $p(x,y)$, and $q(x,y)$ as described. We may assume $(a_i,b_i)_{i<\omega}$ is indiscernible \textcolor{blue}{(this does not follow; see Section \ref{sec:C1})}. Let $r(x_0y_0,x_1y_1)=\tp(a_0b_0,a_1b_1)$. If 
$$
(c_1d_1,c_2d_2,c_3d_3)\models r(x_1y_1,x_2y_2)\cup r(x_2y_2,x_3y_3)\cup r(x_3y_3,x_1y_1),
$$
then we have $d_2d_3\equiv b_0b_1$ and  $c_1\models p(x,d_2)\cup q(x,d_3)$, which is a contradiction. Therefore $(a_ib_i)_{i<\omega}$ witnesses $\SOP_3$.
\end{proof}

Finally, we recall definitions of nondividing and simplicity.

\begin{definition}$~$
\begin{enumerate}
\item Let $a,b$ be tuples in $\M$ and $C\subset\M$. Then $\tp(a/bC)$ \textbf{does not divide over $C$}, written $a\ind^d_C b$, if, for every $C$-indiscernible sequence $(b_i)_{i<\omega}$ with $b_0=b$, there is $a'$ such that $a'b_i\equiv_C ab$ for all $i<\omega$. 

\item $T$ is \textbf{simple} if $\ind^d$ is symmetric in $\M$.
\end{enumerate}
\end{definition}

\begin{fact}\label{divacl}\textnormal{\cite[Remark 5.4]{Adgeo}}  \textcolor{blue}{(This fact is false; see Section \ref{sec:Adler}.)} Given $a,b,C\subset\M$, we have $a\ind^d_C b$ if and only if $\acl(aC)\ind^d_{\acl(C)}\acl(bC)$. 
\end{fact}

We now return to free amalgamation theories. Given a sequence $(b_i)_{i<\mu}$ in $\M$, and some $i<\mu$, we will use the notation $b_{<i}$ to denote $(b_j)_{j<i}$. 

\begin{definition}
Let $\ind$ be a ternary relation on $\M$. Suppose $\mu$ is an ordinal, $(b_i)_{i<\mu}$ is a sequence of tuples, and $C\seq b_0$. Then $(b_i)_{i<\mu}$ is \textbf{$\ind$-independent over $C$} if, for all $i<\mu$, $b_i\equiv_C b_0$ and $b_i\ind_C b_{<i}$.
\end{definition}

Note that if $(b_i)_{i<\mu}$ is $\ind$-independent over a closed set $C$, $b_0$ is closed, and $\ind$ satisfies closure, then $b_{\leq i}$ is closed for all $i<\mu$. We will tacitly use this observation throughout the section. The next result is a key lemma, which says that if $\ind$ is a free amalgamation relation then $\ind$-independent sequences can only witness dividing exemplified by a failure of $\ind^a$.

\begin{lemma}\label{amalgind}
Suppose $\ind$ is a free amalgamation relation for $T$. Fix closed tuples $a$ and $b$ and let $C=a\cap b$ (so $a\ind^a_C b$). Suppose $(b_i)_{i<\mu}$ is $\ind$-independent over $C$, with $b_0=b$. Then there is $a_*$ such that $a_*b_i\equiv_C ab$ for all $i<\mu$.
\end{lemma}
\begin{proof}
By compactness, it suffices to assume $\mu<\omega$. By induction on $n<\mu$, we will find tuples $a_n$ such that $a_nb_i\equiv_C ab$ for all $i\leq n$. For the base case, set $a_0=a$. Assume we have constructed $a_{n-1}$ as required. By full existence, there is $b'\equiv_{a_{n-1}}b_{n-1}$ such that $b'\ind_{a_{n-1}}b_{<n}$. Note that $C\seq a_{n-1}$.

\noindent\textit{Claim}: $a_{n-1}\cap b'b_{<n}=C$.

\noindent\textit{Proof}: First, since $b'\equiv_{a_{n-1}}b_{n-1}$, we have $a_{n-1}\cap b'= a_{n-1}\cap b_{n-1}$. Therefore, it suffices to show $a_{n-1}\cap b_{<n}=C$. For any $i<n$, we have $a_{n-1}b_i\equiv_C ab$ by induction. Therefore, $a\cap b= C$ implies $a_{n-1}\cap b_i= C$.\claim

By the claim and freedom, we have $b'\ind_C b_{<n}$. We also have $b_n\ind_C b_{<n}$ and $b'\equiv_C b_{n-1}\equiv_C b_n$. Therefore $b'b_{<n}\equiv_C b_nb_{<n}$ by stationarity. Let $a_n\in\M$ be such that $a_{n-1}b'b_{<n}\equiv_C a_nb_nb_{<n}$. If $i<n$ then, by induction, $a_nb_i\equiv_C a_{n-1}b_i\equiv_C ab$. We also have $a_nb_n\equiv_C a_{n-1}b'\equiv_C a_{n-1}b_{n-1}\equiv_C ab$. Therefore $a_n$ is as desired.
\end{proof}

Using this, we obtain the following characterization of simplicity for free amalgamation theories.

\begin{theorem}\label{thm:simp}
Given a free amalgamation theory $T$, the following are equivalent.
\begin{enumerate}[$(i)$]
\item $T$ is simple.
\item $T$ is $\NTP_2$.
\item $\ind^d$ and $\ind^a$ coincide in $\M$.
\end{enumerate}
\end{theorem}
\begin{proof}
$(iii)\Rightarrow(i)\Rightarrow(ii)$ is true for any theory (see \cite{KimSST}, \cite{KiKi}).

$(ii)\Rightarrow(iii)$: Suppose $(iii)$ fails. Recall that, in general, $\ind^d$ implies $\ind^a$ and so, using Fact \ref{divacl}, we may fix a closed set $C$ and closed tuples $a,b$ such that $a\cap b=C$ and $a\nind^d_C b$. Let $(b_i)_{i<\omega}$ be a $C$-indiscernible sequence such that $b_0=b$ and, for some $k<\omega$, there is no tuple $a_*$ such that $a_*b_i\equiv_C ab$ for all $i<k$.

Fix a free amalgamation relation $\ind$, and let $b^0_{<\omega}=b_{<\omega}$. Using the full existence axiom, we inductively construct sequences $b^n_{<\omega}$, for $n<\omega$, such that $b^n_{<\omega}\equiv_C b^0_{<\omega}$ and $b^n_{<\omega}\ind_C b^{<n}_{<\omega}$. We show that $a$, $b$, and $b^{<\omega}_{<\omega}$ together witness $\TP_2$ for $T$.

Fix $n<\omega$. Since $b^n_{<\omega}\equiv_C b^0_{<\omega}$, it follows that there is no tuple $a_*$ such that $a_*b^n_i\equiv_C ab$ for all $i<k$. Next, fix a function $\sigma:\omega\func\omega$. By construction and monotonicity, $(b^n_{\sigma(n)})_{n<\omega}$ is $\ind$-independent over $C$. Let $\hat{a}$ be such that $\hat{a}b^0_{\sigma(0)}\equiv_C ab$. Then $\hat{a}\cap b^0_{\sigma(0)}=C$, and so, by Lemma \ref{amalgind}, there is some $a_*$ such that $a_*b^n_{\sigma(n)}\equiv_C\hat{a}b^0_{\sigma(0)}\equiv_C ab$ for all $n<\omega$. 
\end{proof}

Recall that all of our concrete examples of free amalgamation theories are modular, with locally finite algebraic closure. Therefore, we note the following consequence of the previous theorem.

\begin{corollary}
Suppose $T$ is a simple free amalgamation theory. Then $T$ is modular and, if $T$ has locally finite algebraic closure, then $T$ is supersimple.
\end{corollary}
\begin{proof}
Recall that $\ind^d$ satisfies base monotonicity in any theory, and so $T$ is modular by Theorem \ref{thm:simp}. If $T$ has locally finite algebraic closure then, combining Proposition \ref{prop:alg}$(b)$ with condition $(iii)$ of Theorem \ref{thm:simp}, we obtain supersimplicity.
\end{proof}

We can use the results of Section \ref{sec:WEI} to refine these conclusions. Recall that the ternary relation of \emph{nonforking independence} $\ind^f$ is defined by ``forcing extension" on $\ind^d$; precisely, $a\ind^f_C b$ if and only if, for all $\hat{b}\supseteq b$, there is $a'\equiv_{bC} a$ such that $a'\ind^d_C \hat{b}$. Recall also that, for simple theories, $\ind^d$ and $\ind^f$ coincide (see e.g. \cite{KimSST}, \cite{Kibook}). In generalizing important concepts concerning forking in stable theories, Hart, Kim, and Pillay \cite{HKP} introduced hyperimaginaries to define canonical bases and the notion of a $1$-based simple theory.

\begin{definition}
A simple theory $T$ is \textbf{$1$-based} if, for all $A,B\subset\M^{\eq}$, we have $A\ind^f_{\bdd(A)\cap\bdd(B)} B$ in $\M^{\heq}$.
\end{definition}

\begin{fact}\label{thm:1Beq}
If $T$ is simple, with elimination of hyperimaginaries, then the following are equivalent.
\begin{enumerate}[$(i)$]
\item $T$ is $1$-based.
\item $\ind^a$ and $\ind^f$ coincide in $\M^{\eq}$.
\item $T^{\eq}$ is modular.
\end{enumerate}
\end{fact}
\begin{proof}
This is essentially identical to Exercise 3.29 of Adler's thesis \cite{Adthesis}, and we sketch the proof. First, the equivalence of $(i)$ and $(ii)$ follows from elimination of hyperimaginaries and Fact \ref{divacl}. Since $\ind^f$ satisfies base monotonicity in any theory, $(ii)\Rightarrow(iii)$ is trivial. Finally, if $(iii)$ holds then, by Proposition \ref{prop:alg}$(c)$, $\ind^\thrn$ coincides with $\ind^a$ in $\M^{\eq}$, and so $(ii)$ follows from the fact that if $T$ is simple with elimination of hyperimaginaries then $\ind^f$ and $\ind^\thrn$ coincide in $\M^{\eq}$ (see \cite[Theorem 2.8]{EaOn}).
\end{proof}

Altogether, for simple theories with elimination of hyperimaginaries, $1$-basedness expresses that forking in $T^{\eq}$ is as trivial as possible. Unsurprisingly, this has strong consequences for the theory. For example, Kim \cite{KimSST} shows that any simple $1$-based theory, with elimination of hyperimaginaries, satisfies the stable forking conjecture.

We will use Fact \ref{thm:1Beq} to conclude that simple free amalgamation theories are $1$-based. First, we show that under the additional assumption of geometric elimination of imaginaries, conditions $(ii)$ and $(iii)$ of Fact \ref{thm:1Beq} may be checked in $\M$ rather than $\M^{\eq}$. The proof of this only requires the following lemma, which is similar to the techniques in \cite{EaGo}. We could not find a reference for this exact result, and so we outline the proof.

\begin{lemma}\label{GEI}
Suppose $T$ is a complete theory with geometric elimination of imaginaries. Given $e\in\M^{\eq}$, let $g(e)$ be a geometric canonical parameter for $e$ (for $a\in\M$, assume $g(a)=a$). Given $A\subset\M^{\eq}$, let $g(A)=\bigcup\{g(e):e\in A\}$.
\begin{enumerate}[$(a)$]
\item If $A,B,C\subset\M^{\eq}$ then $A\ind^a_C B$ in $\M^{\eq}$ if and only if $g(A)\ind^a_{g(C)}g(B)$ in $\M$.
\item $T$ is modular if and only if $T^{\eq}$ is modular.
\end{enumerate} 
\end{lemma}
\begin{proof}
Part $(a)$. First, note that $\eacl(A)=\eacl(g(A))$ for any $A\subset\M^{\eq}$. Note also that, for any $A,B\subset\M^{\eq}$, we have $g(AB)=g(A)g(B)$ and, if $A\seq B$, then $g(A)\seq g(B)$. Using these observations, we see that $A\ind^a_C B$ in $\M^{\eq}$ if and only if $g(A)\ind^a_{g(C)} g(B)$ in $\M^{\eq}$. So it remains to show $g(A)\ind^a_{g(C)}g(B)$ in $\M^{\eq}$ if and only if $g(A)\ind^a_{g(C)}g(B)$ in $\M$. The forward direction is trivial, so suppose $g(A)\ind^a_{g(C)}g(B)$ in $\M$, and $e\in \eacl(g(AC))\cap\eacl(g(BC))$. Since $g(e)\in\eacl(e)$ and $g(e)$, $g(A)$, $g(B)$, and $g(C)$ are all subsets of $\M$, it follows that $g(e)\seq\acl(g(AC))\cap\acl(g(BC))=\acl(g(C))$. Since $e\in\eacl(g(e))$, we have $e\in \eacl(g(C))$, as desired.

Part $(b)$. Use part $(a)$ to transfer base monotonicity for $\ind^a$ between $\M$ and $\M^{\eq}$ (this uses that $g(A)=A$ for all $A\subset\M$). 
\end{proof}

\begin{theorem}\label{thm:1B}
If $T$ is simple, with elimination of hyperimaginaries and geometric elimination of imaginaries, then the following are equivalent.
\begin{enumerate}[$(i)$]
\item $T$ is $1$-based.
\item $\ind^a$ and $\ind^f$ coincide in $\M$.
\item $T$ is modular.
\end{enumerate}
\end{theorem}
\begin{proof}
The equivalence of $(i)$ and $(iii)$ is immediate from Fact \ref{thm:1Beq} and Lemma \ref{GEI}$(b)$. Since $\ind^f$ satisfies base monotonicity (in $\M$), $(ii)$ implies $(iii)$ is trivial. For $(i)$ to $(ii)$, assume $T$ is $1$-based. Fix $A,B,C\subset\M$ such that $A\ind^a_C B$ in $\M$. Then $A\ind^a_C B$ in $\M^{\eq}$ by Lemma \ref{GEI}$(a)$, and so $A\ind^f_C B$ in $\M^{\eq}$ by Fact \ref{thm:1Beq}. Therefore $A\ind^f_C B$ in $\M$.  
\end{proof}

\begin{corollary}\label{cor:1B}
Any simple free amalgamation theory is $1$-based.
\end{corollary}

It is worth restating this result explicitly for the structures in Example \ref{ex}.\ref{ex1}.

\begin{corollary}\label{cor:FA1B}
If $\cM$ is a countable, simple, ultrahomogeneous structure in a finite relational language $\cL$, whose age is closed under free amalgamation of $\cL$-structures, then $\Th(\cM)$ is $1$-based.
\end{corollary}

In particular, this gives an alternate proof of a recent result of Koponen \cite{Kop1B} showing that the generic tetrahedron-free $3$-hypergraph is $1$-based. We also give this as a partial response to the observation, made in \cite{Kop1B}, that all known examples of countable, simple, ultrahomogeneous structures, in finite relational languages, have $1$-based theories.\footnote{Corollary \ref{cor:FA1B} has also been independently obtained in recent work of Palac\'{i}n \cite{PalGA}.}

Returning to the initial motivations for this section, we have shown that simplicity and $\NTP_2$ coincide for free amalgamation theories. We previously observed that all documented non-simple examples have $\SOP_3$, and so a reasonable conjecture is that simplicity and $\NSOP_3$ also coincide for free amalgamation theories. We will prove this for the class of \emph{modular} free amalgamation theories. Recall that all of our concrete examples of free amalgamation theories are modular, and so it seems quite possible that the modularity assumption is redundant.

\begin{lemma}\label{lem:inddiv}
Suppose $\ind$ satisfies invariance, monotonicity, full existence and stationarity. Then, for any closed tuples $a,b$ and closed sets $C$, with $C\seq a\cap b$, if $a\ind_C b$ then $a\ind^d_C b$ (and hence $a\ind^a_C b$).
\end{lemma}
\begin{proof}
Let $(b_i)_{i<\omega}$ be a $C$-indiscernible sequence, with $b_0=b$. By full existence, there is $a'\equiv_C a$ such that $a'\ind_C b_{<\omega}$. Given $i<\omega$, let $a_i$ be such that $a_ib_i\equiv_C ab$. For all $i<\omega$, we have $a_i\ind_C b_i$ and $a'\ind_C b_i$ by invariance and monotonicity. Since $a'\equiv_C a_i$, we apply stationarity to obtain $a'b_i\equiv_C a_ib_i\equiv_C ab$, as desired.
\end{proof}

\begin{remark}
The reader may have noticed that, so far, none of our results has required transitivity. In fact, we will not use transitivity in any part of this paper. It is included in Definition \ref{def:FAT} in anticipation of its usefulness in future work. For example, one may show that if $\ind$ satisfies invariance, transitivity, and full existence, then $\ind$ satisfies \emph{extension}: for all $a,B,\hat{B},C$, with $B,C$ closed, $B\seq\hat{B}$, and $C\seq a\cap B$, if $a\ind_C B$ then there is $a'\equiv_B a$ such that $a'\ind_C \hat{B}$. Using this, one may prove the version of Lemma \ref{lem:inddiv} obtained by adding transitivity to the assumptions and demanding $a\ind^f_C b$ in the conclusion (see \cite[Theorem 4.1]{CoTe}).
\end{remark}

\begin{theorem}\label{thm:NSOP3}
Suppose $T$ is a modular free amalgamation theory. Then $T$ is simple if and only if $T$ is $\NSOP_3$.
\end{theorem} 
\begin{proof}
First, recall that any simple theory is $\NSOP_3$ (see, e.g., \cite[Claim 2.7]{Sh500}). Conversely, if $T$ is not simple then, in particular, $\ind^d$ does not coincide with $\ind^a$ (this is true for any theory since $\ind^a$ is symmetric). Using Fact \ref{divacl}, we may fix a closed set $C\subset\M$ and closed tuples $a,b$ such that $a\cap b=C$ and $a\nind^d_C b$. Let $(b_i)_{i<\omega}$ be a $C$-indiscernible sequence such that $b_0=b$ and, for some $k<\omega$, there is no tuple $a'$ such that $a'b_i\equiv_C ab$ for all $i<k$. 

\noindent\textit{Claim 1}: We may assume $k=2$.

\noindent\textit{Proof}: First, assume $k>1$ is minimal such that there is no tuple $a'$ with $a'b_i\equiv_C ab$ for all $i<k$. Let $a^*$ be such that $a^*b_i\equiv_C ab$ for all $i<k-1$. For $i<\omega$, let $b^*_i=\acl(b_{i(k-1)}b_{i(k-1)+1}\ldots b_{i(k-1)+k-2})$. Let $b^*=b^*_0$ and $C^*=a^*\cap b^*$, and note that $C\seq C^*$. Suppose, toward a contradiction, that for some $i<j$, there is a tuple $a'$ with $a'b^*_i\equiv_{C^*}a^*b^*\equiv_{C^*}a'b^*_j$. Then, for all $s\in\{i,j\}$ and $t<k-1$, we have $a'b_{s(k-1)+t}\equiv_C a^*b_t\equiv_C ab$. Since $|\{s(k-1)+t:s\in\{i,j\},~t<k-1\}|\geq k$ (recall $i<j$), it follows by indiscernibility that there is a tuple $a''$ such that $a''b_t\equiv_C ab$ for all $t<k$, which contradicts the choice of $k$. 

Now replace $(b^*_i)_{i<\omega}$ with a $C^*$-indiscernible realization of its EM-type over $C^*$, while still assuming $b^*_0=b^*$. \textcolor{blue}{(This need not be possible;  see Section \ref{sec:717}.)} Then $a^*\cap b^*=C^*$, and there is no $a'$ such that $a'b^*_i\equiv_{C^*}a^*b^*$ for all $i<2$.\claim

\noindent\textit{Claim 2}: We may assume $b_0\cap b_1=C$.

\noindent\textit{Proof}: Let $C^*=b_0\cap b_1$ and $a^*=\acl(aC^*)$. Note that $C\seq C^*$ and $(b_i)_{i<\omega}$ is $C^*$-indiscernible. Moreover, there is clearly no $a'$ such that $a'b_i\equiv_{C^*} a^*b$ for all $i<2$. Finally, since $T$ is modular, we have $a^*\cap b=\acl(aC^*)\cap b=\acl((a\cap b)C^*)=C^*$ by Proposition \ref{prop:alg}$(a)$.\claim

Fix a free amalgamation relation $\ind$. By full existence, there is $b^*_0\equiv_a b_0$ such that $b^*_0\ind_a b_0$. By freedom, and since $a\cap b_0= C$, we have $b^*_0\ind_C b_0$. Then $b_0b^*_0\equiv_C b^*_0b_0$ by Lemma \ref{switch}. By Lemma \ref{lem:inddiv}, we have $b_0\ind^a_C b^*_0$ which implies $b_0\cap b^*_0=C$ (recall $C\seq b_0\cap b^*_0$ since $C\seq a$ and $b^*_0\equiv_a b_0$). We inductively construct a sequence $(b^n_1,b^n_2)_{n<\omega}$ such that:
\begin{enumerate}[$(i)$]
\item for all $m\leq n<\omega$ and $i,j\in\{1,2\}$,
$$
b^m_ib^n_j\equiv_C
\begin{cases}
b_0b_1 & \text{if $m<n$, $i=1$, and $j=2$}\\
b_0b^*_0 & \text{otherwise.}
\end{cases}
$$
\item for all $n<\omega$, $b^n_1\ind_C b^{<n}_1b^{< n}_2$;
\item for all $n<\omega$, $b^n_2\ind_C b^{<n}_2b^n_1$.
\end{enumerate}
Let $(b^0_1,b^0_2)=(b_0,b^*_0)$. Suppose we have constructed $(b^i_1,b^i_2)_{i<n}$ as above. By full existence, we may find $b^n_1\equiv_C b_0$ such that $b^n_1\ind_C b^{<n}_1b^{<n}_2$. Then $(ii)$ is immediate. For $(i)$, we want to show $b^i_tb^n_1\equiv_C b_0b^*_0$ for all $i<n$ and $t\in\{1,2\}$. If $i<n-1$ then we have $b^n_1\ind_C b^i_t$, $b^{n-1}_1\ind_C b^i_t$, and $b^n_1\equiv_C b\equiv_C b^{n-1}_1$. By stationarity and induction, $b^i_tb^n_1\equiv_C b^i_tb^{n-1}_t\equiv_C b_0b^*_0$. Now suppose $i=n-1$, and let $s=3-t$. Then $b^n_1\ind_C b^{n-1}_t$, $b^{n-1}_s\ind_C b^{n-1}_t$ (by induction and possibly symmetry), and $b^n_1\equiv_C b\equiv_C b^{n-1}_s$. By stationarity and induction, we have $b^{n-1}_tb^n_1\equiv_C b^{n-1}_tb^{n-1}_s\equiv_C b_0b^*_0$.

Next, we must construct $b^n_2$. First, note that $(b^i_1)_{i<n}$ is $\ind$-independent over $C$ and $b_1\cap b^0_1=b_1\cap b_0= C$. By Lemma \ref{amalgind}, there is $b_*$ such that $b^i_1b_*\equiv_C b^0_1b_1=b_0b_1$ for all $i<n$. Let $B=b^{<n}_1$, which is closed by the closure axiom for $\ind$. By full existence, there is $b^n_2\equiv_B b_*$ such that $b^n_2\ind_B b^{<n}_2b^n_1$. 

\noindent\textit{Claim 3}: $b^{\leq n}_2b^n_1\cap B= C$.

\noindent\textit{Proof}: Fix $i<n$. Then $b^i_1b^n_1\equiv_C b_0b^*_0$ and $b_0\cap b^*_0= C$. It remains to show that, for all $j\leq n$, $b^i_1\cap b^j_2= C$. If $j<n$ then this follows by induction and property $(i)$. For $j=n$, we have $b^i_1b^n_2\equiv_C b^i_1b_*\equiv_C b_0b_1$, and so $b^i_1\cap b^n_2= C$.\claim

By the claim, and freedom, we have $b^n_2\ind_C b^{<n}_2b^n_1$, which gives property $(iii)$. It remains to verify the pertinent parts of property $(i)$. First, we have $b^n_2\ind_C b^n_1$, $b^{n-1}_2\ind_C b^n_1$, and $b^n_2\equiv_C b^{n-1}_2$. By stationarity and choice of $b^n_1$, we have $b^n_1b^n_2\equiv_C b^n_1b^{n-1}_2\equiv_C b_0b^*_0$. Next, if $i<n$ then we have $b^n_2\ind_C b^i_2$, $b^n_1\ind_C b^i_2$, and $b^n_2\equiv_C b^n_1$. By stationarity and choice of $b^n_1$, we have $b^i_2b^n_2\equiv_C b^i_2b^n_1\equiv_C b_0b^*_0$ (recall $b_0b^*_0\equiv_C b^*_0b_0$). Finally, for $i<n$, we have $b^i_1b^n_2\equiv_C b^i_1b_*\equiv_C b_0b_1$. 

This finishes the construction of the sequence $(b^n_1,b^n_2)_{n<\omega}$. Fix $n<\omega$. Define the sequence $(c^n_i)_{i<\omega}$ where, if $i\leq n$ then $c^n_i=b^i_2$, and, if $i>n$ then $c^n_i=b^i_1$. By $(ii)$, $(iii)$, and monotonicity, $(c^n_i)_{i<\omega}$ is $\ind$-independent over $C$. We also have $ac^n_0=ab^*_0\equiv_C ab$ and so $a\cap c^n_0= C$. By Lemma \ref{amalgind}, there is $a_n$ such that $a_nc^n_i\equiv_C ab$ for all $i<\omega$. 

Let $r(x,y)=\tp(a,b/C)$. Recall that, by assumption, $r(x,b_0)\cup r(x,b_1)$ is inconsistent. For $i<\omega$, set $d_i=(b^i_1,b^i_2)$. Fix variables $z=(y_1,y_2)$, and define the types $p(x,z)=r(x,y_1)$ and $q(x,z)=r(x,y_2)$. We use $(a_i)_{i<\omega}$, $(d_i)_{i<\omega}$, $p(x,z)$, $q(x,z)$, and Proposition \ref{SOP3} to show that $T$ has $\SOP_3$.

If $i<j$ then $a_ib^j_1=a_ic^i_j\equiv_C ab$, and so $p(a_i,d_j)$. If $i\geq j$ then $a_ib^j_2=a_ic^i_j\equiv_C ab$, and so $q(a_i,d_j)$. Finally, fix $i<j$. Then
$$
p(x,d_i)\cup q(x,d_j)=r(x,b^i_1)\cup r(x,b^j_2).
$$
By $(i)$, $b^i_1b^j_2\equiv_C b_0b_1$, and so $r(x,b^i_1)\cup r(x,b^j_2)$ is inconsistent.  
\end{proof}

\begin{remark}
By work of Evans and Wong \cite{EvWo}, simplicity coincides with $\NSOP_3$ in the full class of Hrushovski generics $\cM_f$. However, the interesting counterexamples produced by such constructions are often simple and non-modular, and therefore do not fall into our framework.\footnote{For example, if $\cL$ consists of one ternary relation then, with appropriate choice of predimension and control function $f$, $\Th(\cM_f)$ is supersimple and non-modular. See, e.g., \cite[Section 6.2]{Kibook}.}
\end{remark}

 \begin{question}\label{ques:mod}
 Is every free amalgamation theory modular? \textcolor{blue}{(The answer is negative; see Remark \ref{rem:mutchnik}.)}
 \end{question}

\subsection{Simplicity in \Fraisse\ limits with free amalgamation}
For a final application, we take a closer look at simplicity for $\Th(\cM)$, where $\cM$ is countable and ultrahomogeneous, in a finite relational language. Our motivation is the well-known fact that the (binary) generic $K_n$-free graphs are not simple (due to Shelah \cite{Sh500}), while their higher arity analogs, the generic $K^r_n$-free $r$-hypergraphs for $r>2$, are simple (due to Hrushovski \cite{Udibook}).

For the rest of the section, we fix a finite relational language $\cL$. Given $\mathcal{L}$-structures $A$ and $B$ we say $A$ is a \emph{weak substructure} of $B$ if there is an injective map from $A$ to $B$ which preserves the relations in $\mathcal{L}$.

\begin{definition}
Suppose $A$ is an $\cL$-structure. We say that singletons $a_1,\ldots,a_k\in A$ are \textbf{related in $A$} if there is a tuple $\bbar\in A$ such that each $a_i$ is a coordinate of $\bbar$ and $A\models R(\bbar)$ for some relation $R\in\cL$. Given $k\geq 2$, $A$ is \textbf{$k$-irreducible} if any $k$ distinct elements of $A$ are related in $A$.\footnote{This notion usually appears in the literature only for $k=2$ and, in this case, $2$ is omitted.}
\end{definition}

We assume that all classes of finite $\cL$-structures are closed under isomorphism.

\begin{definition}
Suppose $\cF$ is a class of finite $\cL$-structures.
\begin{enumerate}
\item An $\cL$-structure $A$ is \textbf{$\cF$-free} if no weak substructure of $A$ is in $\cF$. Let $\cK_\cF$ denote the class of finite $\cF$-free $\cL$-structures.
\item $\cF$ is \textbf{minimal} if, for any $A\in\cF$, no proper weak substructure of $A$ is in $\cF$. 
\end{enumerate}
\end{definition}

Suppose now that $\mathcal{K}$ is a class of finite $\cL$-structures such that $\cK=\cK_{\cF^*}$ for some class $\cF^*$. Let $\mathcal{F}$ be the class of finite $\cL$-structures $A$ such that $A$ is not in $\cK$, but every proper weak substructure of $A$ is in $\cK$. Then $\cF$ is minimal, and it is straightforward to show that $\cK=\cK_\cF$. We call $\mathcal{F}$ the \emph{minimal forbidden class} for $\mathcal{K}$. The reader may very that, if $\cK$ is a \Fraisse\ class, then $\mathcal{K}$ is closed under free amalgamation if and only if every structure in $\mathcal{F}$ is $2$-irreducible.

\begin{theorem}\label{thm:irred}
Suppose $\cM$ is a countable ultrahomogeneous $\cL$-structure. Let $\cK$ be the age of $\cM$, assume $\mathcal{K}=\mathcal{K}_{\mathcal{F}_*}$ for some class $\mathcal{F}_*$, and let $\mathcal{F}$ be the minimal forbidden class for $\mathcal{K}$.
\begin{enumerate}[$(a)$]
\item If every structure in $\cF$ is $3$-irreducible then $\Th(\cM)$ is simple.
\item Assume $\cK$ is closed under free amalgamation of $\cL$-structures. Then $\Th(\cM)$ is simple if and only if every structure in $\cF$ is $3$-irreducible.
\end{enumerate}
\end{theorem} 
\begin{proof}
Let $\M$ be a monster model of $\Th(\cM)$.

Part $(a)$. Using a straightforward generalization of Hrushovski's proof \cite{Udibook} of simplicity of the generic $K^r_n$-free $r$-hypergraphs for $r>2$, we show $\ind^d$ coincides with $\ind^a$ in $\M$ (which, since $\ind^a$ is symmetric, gives the simplicity of $\Th(\cM)$). First, we show that if $A,B,C$ are pairwise disjoint subsets of $\M$, then $A\ind^d_C B$.

Let $a=(a_1,\ldots,a_n)$ and $b=(b_1,\ldots,b_m)$ be disjoint tuples from $\M$ and fix $C\subset\M$ disjoint from both $a$ and $b$. Fix an infinite $C$-indiscernible sequence $(b^l)_{l<\omega}$, with $b^0=b$. We want to find $a'=(a'_1,\ldots,a'_n)$ such that $a'b^l\equiv_C ab$ for all $l<\omega$. By passing elements from $b$ to $C$, we may assume $b^0\cap b^1=\emptyset$.

Let $E=C\cup \bigcup_{l<\omega}b^l$. Define an $\cL$-structure $D$ with universe $a'E$ where $a'=(a'_1,\ldots,a'_n)$ is a tuple disjoint from $E$. Define relations on $D$ so that $E$ is a substructure of $D$ and, for each $l<\omega$, $a'b^l\cong_C ab$. No other relations hold in $D$; in particular, if $a'_i\in a'$, $b^l_j\in b^l$, and $b^m_k\in b^m$, with $l\neq m$, then $a',b^l_j,b^m_k$ are not related in $D$.

Note that, if $D$ is $\cF$-free, then we may embed $D$ in $\M$ over $E$, and the image of $a'$ in $\cM$ is as desired. Therefore it suffices to show that $D$ is $\cF$-free.

Suppose, toward a contradiction, that some $A\in\cF$ is a weak substructure of $D$. Since $E$ is $\cF$-free, we must have some $a'_i\in A\cap a'$. Moreover, for any fixed $l<\omega$, we have $a'b^l\cong_Cab$, and so, since $ab C$ is $\cF$-free, it follows that $A$ is not entirely contained in any single $a'b^l C$. Therefore, we may fix $l<m<\omega$, and elements $b^l_j\in b^l\backslash b^m$ and $b^m_k\in b^m\backslash b^l$, such that $b^l_j,b^m_k\in A$. Since $A$ is $3$-irreducible, it follows that $a'_i,b^l_j,b^m_k$ are related in $D$, which is a contradiction.

Now suppose $A,B,C\subset\M$ are arbitrary with $A\ind^a_C B$. Let $A'=\acl(AC)\backslash\acl(C)$, $B'=\acl(BC)\backslash\acl(C)$, and $C'=\acl(C)$. Then $A',B',C'$ are pairwise disjoint, and so $A'\ind^d_{C'}B'$. Then $A'C'\ind^d_{C'}B'C'$, and so $A\ind^d_C B$ by Fact \ref{divacl}. 

Part $(b)$. Assume $\cM$ is closed under free amalgamation of $\cL$-structures, and suppose $A\in\cF$ is not $3$-irreducible. We show that $\ind^a$ does not coincide with $\ind^d$ in $\M$, which, by Theorem \ref{thm:simp}, suffices to show that $\Th(\cM)$ is not simple.

Enumerate $A=\{a_1,\ldots,a_n\}$ so that $a_1,a_2,a_3$ are not related in $A$. Let $\hat{a}=(a_4,\ldots,a_n)$. We define an $\cL$-structure $E$ with universe $\{b_4,\ldots,b_n\}\cup\bigcup_{l<\omega}\{b^l_2,b^l_3\}$ and define relations such that, setting $\hat{b}=(b_4,\ldots,b_n)$:
\begin{enumerate}
\item $b^l_t\hat{b}\cong a_t\hat{a}$ for all $l<\omega$ and $t\in\{2,3\}$;
\item $b^l_2b^m_3\hat{b}\cong a_2a_3\hat{a}$ for all $l<m<\omega$;
\item no other relations hold in $E$.
\end{enumerate}

Suppose, toward a contradiction, that some $A'\in\cF$ is a weak substructure of $E$. Recall that every element of $\cF$ is $2$-irreducible. By construction of $E$, it follows that $A'$ is a substructure of $b^l_2b^m_3\hat{b}$ for some $l<m<\omega$. But $b^l_2b^m_3\hat{b}$ is isomorphic to a proper substructure of $A$ by definition, which contradicts that $\cF$ is minimal. Therefore $E$ is $\cF$-free and so we may assume $E\subset\M$. Note that $(b^l_2,b^l_3)_{l<\omega}$ is $\hat{b}$-indiscernible. 

Let $b_2=b^0_2$ and $b_3=b^0_3$. Since $\cF$ is minimal, we may use similar arguments to find $b_1\in\M$ such that $b_1b_t\hat{b}\cong a_1a_t\hat{a}$, for $t\in\{2,3\}$. We use $(b^l_2,b^l_3)_{l<\omega}$ to show $b_1\nind^d_{\hat{b}} b_2b_3$ (since algebraic closure is trivial, we have $b_1\ind^a_{\hat{b}}b_2b_3$, and so this suffices to finish the proof). Suppose, toward a contradiction, there is $b_*\in\M$ such that $b_*b^l_2b^l_3\hat{b}\cong b_1b_2b_3\hat{b}$ for all $l<\omega$. Then, by construction, $b_*b^0_2\hat{b}\cong a_1a_2\hat{a}$, $b_*b^1_3\hat{b}\cong a_1a_3\hat{a}$, and $b^0_2b^1_3\hat{b}\cong a_2a_3\hat{a}$. Since $a_1,a_2,a_3$ are not related in $A$, it follows that $A$ is a weak substructure of $b_*b^0_2b^1_3\hat{b}$, which contradicts that $\M$ is $\cF$-free. 
\end{proof}

\articleend

%\author{Gabriel Conant}

\title{Corrections to ``An Axiomatic Approach to Free Amalgamation"}

\author{Gabriel Conant}
\address{Department of Mathematics\\
The Ohio State University\\
Columbus, OH, USA}
\email{conant.38@osu.edu}

\date{\today}

\maketitle

Recent work of Scott Mutchnik \cite{Mutch2,Mutch1} has led to renewed interest in the previous paper, and the class of free amalgamation theories. The purpose of this note is to correct a number of errors in the paper. In the sequel, sections are numbered starting with 8 in order to not conflict with references to  Sections 1 through 7 of the original paper above.

I would like to thank  Michele Bailetti for bringing the issues in Section \ref{sec:TP2} and Section \ref{sec:717} to my attention. The issue in Section \ref{sec:Adler} arose from discussions with Alex Kruckman, who also deserves very special thanks for fixing the issue in Section \ref{sec:717} by providing the proof of Proposition \ref{prop:717}. Thanks also to Baraa Abd Aldaim and Katie Ellman-Aspnes for correcting typos in an earlier draft.

\section{Indiscernible oversights}\label{sec:C1}

\subsection{TP\textsubscript{2}}\label{sec:TP2} 
Definition \ref{def:TP2} gives a nonstandard formulation of $\TP_2$, which is incorrect and, in particular, too weak. However, the witness to $\TP_2$ constructed in the proof of Theorem \ref{thm:simp} satisfies some additional indiscernibility assumptions (described in the next proposition), which suffice to overcome the discrepancy.

\begin{proposition}
 Suppose there are tuples $a,b\in\M$ and  an array $(b^m_n)_{m,n<\omega}$ in $\M$ satisfying the following properties.
 \begin{enumerate}[$(i)$]
 \item There does not exist a tuple $a_*$ such that $a_*b^0_{n}\equiv ab$ for all $n<\omega$.
 \item For all $\sigma\colon\omega\to\omega$ there is a tuple $a_*$ such that $a_*b^m_{\sigma(m)}\equiv ab$ for all $m<\omega$.
 \item The sequence $(b^0_i)_{i<\omega}$ is indiscernible and, for all $m<\omega$, $b^m_{<\omega}\equiv b^0_{<\omega}$.
 \end{enumerate}
 Then $T$ has $\TP_2$.
 \end{proposition}
 \begin{proof}
 Note that $b^m_n\equiv b$ for all $m,n<\omega$ by condition $(ii)$. Let $p(x,b)=\tp(a/b)$. Then by condition $(i)$, $\bigcup_{n<\omega}p(x,b^0_n)$ is inconsistent. By compactness there is a formula $\varphi(x,b)\in p$ and some finite $I\subset\omega$ such that $\{\varphi(x,b^0_n):n\in I\}$ is inconsistent. Since $b^0_{<\omega}$ is indiscernible, it follows that $\{\varphi(x,b^0_n):n<\omega\}$ is $k$-inconsistent where $k=|I|$. Therefore, for any $m<\omega$, $\{\varphi(x,b^m_n):n<\omega\}$ is $k$-inconsistent since $b^m_{<\omega}\equiv b^0_{<\omega}$. With $(ii)$, $\varphi(x,y)$ now satisfies the standard definition of $\TP_2$. 
 \end{proof}
 
Note that (the incorrect) Definition \ref{def:TP2} includes only condition $(ii)$ and a stronger form of $(i)$ (which is implied by $(i)$ and $(iii)$). In any case, it is clear in the proof of Theorem \ref{thm:simp} that we have built an array satisfying $(i)$, $(ii)$, and $(iii)$. It is also worth pointing out that $(iii)$ is a weak form of array indiscernibility, and thus the converse of the previous proposition holds as well. Further discussion and details can be found on MathStackExchange  \cite{MichMSE}.
 
\subsection{SOP\textsubscript{3}} In the characterization of SOP$_3$ given by Proposition \ref{SOP3}, one must assume that $(b_i)_{i<\omega}$ is indiscernible in order to be able to lift indiscernibility to the pair sequence $(a_i,b_i)_{i<\omega}$ using EM-types. Alternatively, it would suffice to replace the types $p$ and $q$ by formulas. The only part of the paper that uses Proposition \ref{SOP3} is the construction of SOP$_3$ in the proof of Theorem \ref{thm:NSOP3}. 
 In the proof, we have a type $r(x,y)$ over a parameter set $C$, as well as some $b_0\equiv_C b_1$ such that $r(x,b_0)\cup r(x,b_1)$ is inconsistent. Then, with $z=(x_1,x_2)$, we use $p(x,z)=r(x,y_1)$ and $q(x,z)=r(x,y_2)$ in the application of Proposition \ref{SOP3}. However, by compactness, there is a formula $\varphi(x,y)$ in $r(x,y)$ such that $\varphi(x,b_0)\wedge\varphi(x,b_1)$ is inconsistent, and  we can replace $r$ with $\varphi$ in the rest of the proof.

\section{Algebraic closure in dividing independence}\label{sec:Adler}

Fact \ref{divacl}  states that $a\indiv_C b$ if and only if $\acl(aC)\indiv _{\acl(C)}\acl(bC)$. The justification for this is \cite[Remark 5.4(3)]{Adgeo2}, which was recently discovered   to be false  (see \cite{CoKr2}). \textbf{Therefore, Fact \ref{divacl}  is  incorrect.} Fortunately, all uses of this erroneous result can be replaced by weaker statements, which are true.

\begin{proposition}\label{fixacl}$~$
\begin{enumerate}[$(a)$]
\item If $\acl(aC)\indiv_{\acl(C)}\acl(bC)$ then $a\indiv_C b$.
\item If $a\indiv_C b$ then $a\inda_C b$.
\item If $\indiv$ and $\inda$ do not coincide then there is an algebraically closed set $C$ and algebraically closed tuples $a,b$ such that $a\cap b=C$ and $a\nindiv_C b$.
\item Fact \ref{divacl}  holds under the additional assumption that $T$ is simple.
\end{enumerate}
\end{proposition}
\begin{proof}
Part $(a)$. This is a straightforward exercise. See also \cite[Section 2.2]{CoKr2}.
%Suppose $\acl(aC)\indiv_{\acl(C)}\acl(bC)$. Then $a\indiv_{\acl(C)} b$ by monotonicity. To prove $a\indiv_C b$, fix a $C$-indiscernible sequence $(b_i)_{i<\omega}$ with $b_0=b$. Then $(b_i)_{i<\omega}$ is $\acl(C)$-indiscernible (see \cite[Corollary 1.7]{Cabook2}), hence there is some $a^*$ such that $a^*b_i\equiv_{\acl(C)} ab_i$ for all $i<\omega$ (so, in particular, $a^*b_i\equiv_C ab_i$ for all $i<\omega$).

Part $(b)$. This is a folklore fact, which is often cited in the literature as a corollary of the incorrect remark from \cite{Adgeo2}. See  \cite[Section 2]{CoHa} for  discussion and a direct proof. 

Part $(c)$. Assume $\indiv$ and $\inda$ do not coincide. By part $(b)$, there are $a',b',C'$ such that $a'\nindiv_{C'} b'$ and $a'\inda_{C'} b'$. Let $a$ enumerate $\acl(a'C')$, $b$ enumerate $\acl(b'C')$, and $C=\acl(C')$. Then $a'\inda_{C'}b'$ says precisely that $a\cap b=C$. Moreover, since $a'\nindiv_{C'} b'$, we have $a\nindiv_C b$ by part $(a)$.

Part $(d)$. Recall that if $T$ is simple then $\indiv$ coincides with $\indf$. Moreover, the analogue of Fact \ref{divacl} for $\indf$ in a simple theory is well-known (see, e.g., \cite[Proposition 5.20]{Cabook2}, whose proof is left as an exercise). It is also not hard to show that $\indf$ satisfies the analogue of Fact \ref{divacl} in \emph{any} theory (see \cite[Proposition 2.12]{CoKr2}). 
\end{proof}

We now provide patches for each use of  Fact \ref{divacl} in the paper.

\begin{enumerate}[$(1)$]
\item In the proof of Theorem \ref{thm:simp}$[(ii)\Rightarrow(iii)]$, our use of Fact \ref{divacl} was only in order to apply Proposition \ref{fixacl}$(c)$. 

\item In Fact \ref{thm:1Beq}, we assume $T$ is simple. Thus Fact \ref{divacl} holds for $T$ by Proposition \ref{fixacl}$(d)$. 

\item In Lemma \ref{lem:inddiv}, we tacitly used Fact \ref{divacl}, but only in order to conclude Proposition \ref{fixacl}$(b)$.

\item In the proof of Theorem \ref{thm:NSOP3}, our use of Fact \ref{divacl} was only in order to apply Proposition \ref{fixacl}$(c)$.

\item In the proof of Theorem \ref{thm:irred}, our use of Fact \ref{divacl} was only in order to apply Proposition \ref{fixacl}$(a)$. 
\end{enumerate}

\section{The dichotomy for modular free amalgamation theories}\label{sec:717}

Theorem \ref{thm:NSOP3}  states that any modular free amalgamation theory $T$ is either simple or  $\SOP_3$. However, there is a  gap in the proof of Claim 1. In order to discuss the details, let us first recall how the proof starts. Assume $T$ is not simple. Then $\inda$ and $\indiv$ do not coincide by  Theorem \ref{thm:simp}. By Proposition \ref{fixacl}$(c)$, there are algebraically closed tuples $a$, $b$ such that, setting $C=a\cap b$, we have $a\nindiv_C b$. So there is a $C$-indiscernible sequence $(b_i)_{i<\omega}$ such that, if $p(x,y)=\tp(a,b/C)$, then $\{p(x,b_i):i<\omega\}$ is $k$-inconsistent for some $k<\omega$. At this point, Claim 1 attempts to reduce to the case that $k=2$. We do this by choosing a minimal  $k$ witnessing $a\nindiv_C b$, and then defining the sequence $b^*_i=\acl(b_{i(k-1)}b_{i(k-1)+1}\ldots b_{i(k-1)+k-2})$. By minimality, there is some $a^*$ such that $a^*b_i\equiv_C ab$ for all $i<k-1$, and it is (correctly) shown that  $\{p^*(x,b^*_i):i<\omega\}$ is $2$-inconsistent, where $p^*(x,y)=\tp(a^*,b^*_0/C)$. However, we may have lost $a^*\cap b^*_0=C$. To fix this, the reader is told to set $C^*=a^*\cap b^*_0$  and replace $(b^*_i)_{i<\omega}$ with a $C^*$-indiscernible realization of its EM-type, while also moving by an automorphism to ensure we keep $b^*_0$ as the first term of the sequence. But in order for this to work, we would need to know that $b^*_i\equiv_{C^*} b^*_j$ for all $i<j<\omega$, and there is no reason for this to be true. 

Before getting into the details of how to fix this gap,  we first note that if one strengthens the modularity assumption to \emph{disintegration} of algebraic closure, then the above proof works. Indeed, each $b_i$ is algebraically closed, and thus disintegration would imply $b^*_i=b_{i(k-1)}\ldots b_{i(k-1)+k-2}$. Thus, since $a\cap b=C$ and $a^*b_i\equiv_C ab$ for all $i<k-1$, we do have $a_*\cap b^*_0=C$ in the disintegrated case.  The reason to point this out is that all of the examples of free amalgamation theories  above have disintegrated algebraic closure (but see Remark \ref{rem:mutchnik} below for  further details). 

We now turn to fixing the gap.
In \cite{Mutch1}, Mutchnik used the general proof strategy of Theorem \ref{thm:NSOP3} to show that SOP$_1$ and SOP$_2$ coincide at the level of theories. The overall scope of this result draws  from many areas of neostability  (e.g., \cite{ChKa} and \cite{KaRam}). However, in the special case of modular free amalgamation theories, Mutchnik's strategy in the analogue of Claim 1 suggests an elementary fix for the gap  above. We also note that, while modularity of $T$ is used in the proof of Claim 2 of Theorem \ref{thm:NSOP3}, it is not used in the (incorrect) proof of Claim 1. So we emphasize that the following correct proof does require modularity. 

\begin{proposition}\label{prop:717}
Let $T$ be a modular complete theory such that $\inda$ and $\indiv$ do not coincide. Then there is an algebraically closed set $C$, algebraically closed tuples $a$ and $b$, and a $C$-indiscernible sequence $(b_i)_{i<\omega}$, with $b_0=b$, such that $a\cap b=C$ and there is no $a^*$ with $a^{*}b_0\equiv_C a^{*}b_1\equiv_C ab$. 
\end{proposition}
\begin{proof}
We first argue that it suffices to find  an algebraically closed set $C$, a $C$-indiscernible sequence $(b'_i)_{i<\omega}$,  and some tuple $a'$ such that 
\begin{equation*}
\textnormal{$\textstyle a'\inda_C b'_0$ and there is no $a^*$ with $a^*b'_0\equiv_C a^*b'_1\equiv_C a'b'_0$.}\tag{$\dagger$}
\end{equation*}
Indeed, with $(\dagger)$ in hand,  let $a$ enumerate $\acl(a'C)$ and $b$ enumerate $\acl(b'_0C)$. So $a\cap b=C$. Let $b_0=b$. For $i>0$, choose an automorphism $\sigma_i$ over $C$ such that $\sigma_i(b'_0)=b'_i$ and set $b_i=\sigma_i(b_0)$. 
Let $(b^*_i)_{i<\omega}$ be a $C$-indiscernible realization of the EM-type of $(b_i)_{i<\omega}$ over $C$. Let $b''_i$ be the subtuple of $b^*_i$ corresponding to the location of $b'_i$ in $b_i$.  Since each $b_i$ realizes $\tp(b_0/C)$, so does each $b^{*}_i$ and, in particular, $b^{*}_i$ enumerates $\acl(b''_i)$. Since $(b'_i)_{i<\omega}$ is $C$-indiscernible, we have $b''_{<\omega}\equiv_C b'_{<\omega}$. Thus after moving by an automorphism over $C$, we may assume $(b_i)_{i<\omega}$ is $C$-indiscernible. By $(\dagger)$, we clearly have no $a^*$ such that $a^{*}b_0\equiv_C a^{*}b_1\equiv_C ab$.

Now we construct $C$, $(b'_i)_{i<\omega}$, and $a'$ satisfying $(\dagger)$.\footnote{This proof is due to Alex Kruckman.} Fix $a$, $b$, and $C$ such that $a\inda_C b$ and $a\nindiv_C b$ (recall Proposition \ref{fixacl}$(b)$). By Proposition \ref{fixacl}$(a)$ (or $(c)$), we may assume $C$ is algebraically closed. Fix a $C$-indiscernible sequence $(b_i)_{i<\omega}$, with $b_0=b$, such that there is no $a^*$ satisfying $a^*b_i\equiv_C ab$ for all $i<\omega$. Since $a\inda_C b$, we may  choose $k\geq 1$  maximal such that there is a tuple $a_0$ with $a_0\inda_C b_{<k}$ and $a_0b_t\equiv_C ab$ for all $t<k$.  
Set $b'_i=b_{ik}b_{ik+1}\ldots b_{ik+k-1}$. Then $(b'_i)_{i<\omega}$ is $C$-indiscernible, and we have $a_0\inda_C b'_0$. 
Let $\kappa=|\acl(b_{\leq k}C)|^+$ and, using extension for $\inda$, construct a sequence $(a_i)_{i<\kappa}$ such that $a_i\equiv_{Cb'_0}a_0$ and $a_i\inda_C a_{<i}b'_0$. By base monotonicity for $\inda$, we have $a_i\inda_{a_{<i}C}b'_0$ for all $i<\kappa$. Using left transitivity and induction, we see that  $a_{<i}\inda_C b'_0$ for all $i<\kappa$, hence $a_{<\kappa}\inda_C b'_0$.

Finally, we show that there is no $a^*_{<\kappa}$ with $a^*_{<\kappa}b'_0\equiv_C a^*_{<\kappa}b'_1\equiv_C a_{<\kappa}b'_0$, and so we can set $a'=a_{<\kappa}$ to obtain $(\dagger)$. Toward a contradiction, suppose that there is such an $a^*_{<\kappa}$. Fix some $i<\kappa$. Then $a^*_ib'_0\equiv_C a_ib'_0\equiv a_0b'_0$,
and so for all $t<k$, we have $a^*_ib_t\equiv_C a_0b_t\equiv_C ab$. Also, $a^*_ib'_1\equiv_C a_ib'_0\equiv_C a_0b'_0$, and so $a^*_ib_{k}\equiv_C a_0b_0\equiv_C ab$. So $a^*_ib_k\equiv_C ab$ for all $t\leq k$ and, by maximality of $k$, it follows that $a^*_i\ninda_C b_{\leq k}$. 

Now, for each $i<\kappa$, fix a witness $e_i\in (\acl(a^*_iC)\cap \acl(b_{\leq k}C))\backslash C$ to $a^*_i\ninda_C b_{\leq k}$. By choice of $\kappa$, there are $j<i<\kappa$ such that $e_i=e_j$, and so $a^*_i\ninda_C a^*_j$. But $a^*_ia^*_j\equiv_C a_ia_j$, and $a_i\inda_C a_j$, which is a contradiction.
\end{proof}

\begin{remark}\label{rem:mutchnik}
As mentioned above, the previous proof is inspired by work of Mutchnik \cite{Mutch1}, who uses a strategy similar to Theorem \ref{thm:NSOP3} to construct SOP$_3$ in any theory with SOP$_1$ and  NSOP$_2$  (hence there are no such theories, so SOP$_1$ and SOP$_2$ coincide). In subsequent work, Mutchnik \cite{Mutch2} uses similar tools to prove several new results about free amalgamation theories. Recall that all of the examples  of such theories given above are modular and, in fact, disintegrated. Moreover, any \emph{simple} free amalgamation theory is modular by Corollary \ref{cor:1B}. Thus, Question 7.19  asks if \emph{every} free amalgamation theory is modular. In Section 4 of \cite{Mutch2}, Mutchnik provides a negative answer by showing that the theory $T_{f,c}$ of a generic binary function with a distinguished constant is a non-modular free amalgamation theory. By work of Kruckman and Ramsey (see Corollary 3.13 and Proposition 3.14 of \cite{KrRam}), $T_{f,c}$ is not simple, but is NSOP$_3$ (in fact, NSOP$_1$). Altogether,  the modularity assumption cannot be removed from Theorem \ref{thm:NSOP3}. However, Mutchnik  extends Theorem \ref{thm:NSOP3} to the non-modular case by proving the following results for an arbitrary free amalgamation theory $T$ (see Section 2 of  \cite{Mutch2}).
\begin{enumerate}[$(1)$]
\item $T$ is either NSOP$_1$ or SOP$_3$.
\item Moreover, if $T$ is NSOP$_1$ then Kim-independence coincides with algebraic independence over models. 
\item Therefore, $T$ is simple if and only if it is NSOP$_1$ and modular.
\end{enumerate}
We also note that the theory $T_{f,c}$ is not rosy (by Propositions 3.17 and 3.18 of \cite{KrRam}), which shows the necessity of the modularity assumption in Corollary \ref{cor:FArosy}.
\end{remark}

\end{document}